\documentclass[10pt]{article} 
\usepackage{mathtext}
\usepackage{amsmath}
\usepackage{amsfonts}
\usepackage{amssymb}
\usepackage{amsthm}
\usepackage{mathrsfs}
\usepackage{bbm}
\usepackage{bbold}
\usepackage{graphicx}
\usepackage{subcaption}
\usepackage{listings}                 
\usepackage[english]{babel}
\usepackage{xcolor}

\usepackage{geometry} 
\newtheorem{theorem}{Theorem}[section]
\newtheorem{assumption}{Assumption}

\newtheorem{lemma}{Lemma}[section]
\newtheorem{example}{Example}[section]

\newtheorem{remark}{Remark}[section]

\newcommand{\E}{\ensuremath{\mathbb E}}

\newcommand{\N}{\ensuremath{\mathbb N}}
\newcommand{\p}{\ensuremath{\mathbb P}}
\newcommand{\1}{\ensuremath{\mathbb 1}}

\DeclareMathOperator*{\esssup}{ess\,sup}
\DeclareMathOperator*{\argmin}{arg\,min}

\title{Convergence of the Kiefer-Wolfowitz algorithm in the presence of discontinuities\thanks{Both 
authors were supported by the ``Lend\"ulet'' grant LP 2015-6. Part of this research
was performed while the authors were participating
in the Simons Semester on Stochastic Modeling and Control
at Banach Center, Warsaw in 2019.}} 

\author{Mikl\'os R\'asonyi\thanks{Alfr\'ed R\'enyi Institute of Mathematics, Budapest
and Mathematical Institute, Warsaw.} \and 
Kinga Tikosi\thanks{Alfr\'ed R\'enyi Institute of Mathematics, Budapest
and Mathematical Institute, Warsaw. During the preparation
of this paper the author attended the PhD school of 
Central European University, Budapest and was also affiliated with E\"otv\"os Lor\'and University, 
Budapest, where she was supported by Project no. ED 18-1-2019-0030 (Application domain specific highly 
reliable IT solutions subprogramme) which has been implemented with the support provided from the National 
Research, Development and Innovation Fund of Hungary, financed under the Thematic Excellence Programme funding scheme.} }

\date{\today}

\begin{document}

\maketitle

\begin{abstract}
	In this paper we estimate the expected error of a stochastic approximation
	algorithm where the maximum of a function is found using finite differences of
	a stochastic representation of that function. An error estimate of the order $n^{-1/5}$ for the $n$th iteration
	is achieved using suitable parameters. The novelty with respect to previous studies is
	that we allow the stochastic representation to be discontinuous and to consist of possibly dependent random variables (satisfying a mixing condition).  
\end{abstract}

\begin{section}{Introduction}

We are interested in maximizing a function $U:\mathbb{R}^{d}\to\mathbb{R}$ which is unknown. 
However, we can observe a sequence $J(\theta,X_{n})$, $n\geq 1$ where
$J:\mathbb{R}^d\times\mathbb{R}^m\to\mathbb{R}$ is measurable,
\begin{equation}\label{rubayyat}
\E[J(\theta, X_1)]= U(\theta),\ \theta\in\mathbb{R}^{d},
\end{equation}
and $X_{n}$, $n\geq 1$ is an $\mathbb{R}^m$-valued stationary process in the strong sense. 
The stochastic representations $J(\theta,X_{n})$ 
are often interpreted as noisy measurements of $U(\theta)$. In this paper we focus on applications
to mathematical finance, described in Section \ref{sec:mathfin} below, where $J(\theta,X_t)$
are functionals of observed economic variables $X_t$ and $\theta$ determines an investor's
portfolio strategy. In that context, stochasticity does not come from measurement errors
but it is an intrinsic property of the system. 
Maximizing $U$ serves to find the best investment policy in an online,
adaptive manner. 

We study a recursive algorithm employing finite differences, 
as proposed by Kiefer and Wolfowitz in \cite{kiefer1952stochastic}.
This is a variant of the Robbins-Monro stochastic gradient method \cite{robbins1951stochastic} where, instead of the objective function itself, its gradient is 
assumed to admit a stochastic representation. 
 
The novelty in our work is that we do not assume differentiability, not even continuity of $\theta\to J(\theta,\cdot)$ and the sequence $X_n$ may well be dependent as long as it satisfies a mixing condition. The only result in such a
setting that we are aware of is in \cite{laruelle2012stochastic}, however, they only study almost sure convergence, without
a convergence rate. Our purpose is not to find the weakest possible hypotheses but to arouse keen
interest in the given problem that can lead to
further, more general results. Our work is also a continuation of \cite{fort,chau2016fixed}, where discontinuous
stochastic gradient procedures were treated.

The main theorems are stated in Section \ref{sett} and
proved in Section \ref{prof}. Section \ref{auxiliary} recalls earlier results that we are relying on.
A numerical example is provided in Section \ref{opt}. We explain the
significance of our results for
algorithmic trading in Section \ref{sec:mathfin}. 

\begin{section}{Setup and results}\label{sett}


For real-valued quantities $X,Y$ the notation $O(X)=Y$ means that there is a constant $C>0$ such that $|X|\leq CY$.
We will always work on a fixed probability space $(\Omega,\mathcal{F},P)$ equipped with a filtration $\mathcal{F}_{n}$, $n\in\mathbb{N}$ such that $\mathcal{F}_{0}=\{\emptyset,\Omega\}$. A decreasing sequence of sigma-algebras $\mathcal{F}^{+}_{n}$, $n\in\mathbb{N}$ is also given, such that, for each $n\in\mathbb{N}$, 
$\mathcal{F}_n$ and $\mathcal{F}_n^+$ are independent and
$X_n$ is adapted to $\mathcal{F}_n$. The notation $\E[X]$ refers to the
expectation of a real-valued random variable $X$, while $\E_k[X]$ is a shorthand
notation for $\E[X\vert\mathcal{F}_k]$, $k\in\mathbb{N}$.
$P_k(A)$ refers to the conditional probability $P(A\vert\mathcal{F}_k)$.
We denote by $\1_A$ the indicator 
of a set $A$. The notation $\omega$ refers to a generic
element of $\Omega$. For $r\geq 1$, we refer to the set of random variables
with finite $r$th moments as $L^r$. $|\cdot|$
denotes the Euclidean norm in $\mathbb{R}^k$
where $k$ may vary according to the context.

For $i=1,\ldots,d$, let $\mathbf{e}_{i}\in\mathbb{R}^d$ denote the vector whose $i$th coordinate is $1$ and the other coordinates are $0$. For two vectors $v,w\in\mathbb{R}^m$ the relation $v\leq w$ expresses that $v^i\leq w^i$ for all the components
$i=1,\ldots,m$. Let $B_r:=\{\theta\in\mathbb{R}^d:\,
|\theta|\leq r\}$ denote the ball of radius $r$, for
$r\geq 0$.

Let the function $U:\mathbb{R}^d\to\mathbb{R}$ have a unique maximum at the point $\theta^{*}\in\mathbb{R}^{d}$.
Consider the following recursive stochastic approximation scheme for finding $\theta^{*}$:
\begin{align} \label{recursive_scheme}
    \theta_{k+1}=\theta_k + \lambda_k H(\theta_k,X_{k+1},c_k)\text{, for }k\in\mathbb{N},
\end{align} starting from some initial (deterministic) guess $\theta_0\in\mathbb{R}^d$, 
where $H$ is an estimator of the gradient of $J$, defined as $$ 
H (\theta,x,c)=\sum_{i=1}^d\frac{J(\theta+c\mathbf{e}_i,x)
-J(\theta-c\mathbf{e}_i,x)}{2c} \mathbf{e}_i,
$$
for all $\theta\in\mathbb{R}^d$, $x\in\mathbb{R}^m$ and $c>0$.

The sequences $(\lambda_k)_{k\in\mathbb{N}}$ and $(c_k)_{k\in\mathbb{N}}$ appearing in \eqref{recursive_scheme} will consist of
positive real numbers, which are to be specified later. We will distinguish the cases where $\lambda_k$, $c_{k}$ tend 
to zero and where they are kept constant, the former being called \emph{decreasing gain} approximation and the latter \emph{fixed gain} approximation.

\begin{remark} \label{rem:X_same_or_different}
{\rm Our results below could easily be formulated in a more general setting where
$J(\theta_k+c_k\mathbf{e}_i,X_{k+1}(i))$ and $J(\theta_k-c_k\mathbf{e}_i,X_{k+1}'(i))$, $i=1,\ldots,d$ are considered
with distinct $X_{k+1}(i)$ and $X_{k+1}'(i)$. In the applications that motivate us
this is not the case hence, for reasons of simplicity, 
we stay in the present setting.}
\end{remark}

\begin{assumption} \label{asp:G}
    $U$ is continuously differentiable with unique maximum $\theta^*\in\mathbb{R}^d$.
    Denote $G(\theta)=\nabla U(\theta)$. 
    The function $G$ is assumed Lipschitz-continuous with Lipschitz-constant $L_G$.  
\end{assumption}

We assume in the sequel that the function $J$ in \eqref{rubayyat} has a specific form. 
Note that though $J$ is not continuous, $U$ can nonetheless be continuously differentiable, by the
smoothing effect of randomness.

\begin{assumption} \label{asp:H-special_form} 
Let the function $J$ be of the following specific form:
$$
J(\theta,x)=l_{0}(\theta)\mathbb{1}_{A_{0}(x)}+{}
\sum _{i=1}^{m_s}  \mathbb{1}_{A_{i}(x)} l_i (\theta,x),{}
$$
where $l_{i}:\mathbb{R}^{d}\times\mathbb{R}^{m}\to\mathbb{R}^{d}$
are Lipschitz-continuous (in both variables) for $i=1,\ldots,m_{s}$ and, for some
$m_p,m_p'\in\mathbb{N}$,
$$
A_{i}(x):=\left(\cap_{j=1}^{m_{p}} \{\theta: x\leq g_i^j(\theta)\}\right)
\bigcap \left(\cap_{j=1}^{m_{p}'} \{\theta: x> h_i^j(\theta)\}\right),\ i=1,\ldots,m_{s}
$$
with Lipschitz-continuous functions $g^{j}_{i},h_i^j:\mathbb{R}^d\to\mathbb{R}^m$.
Furthermore, $A_{0}(x):=\mathbb{R}^{d}\setminus \cup_{i=1}^{m_{s}}
A_{i}(x)$ and $$
\cup_{x\in\mathbb{R}^{m}}\cup_{i=1}^{m_{s}} A_{i}(x)\subset B_{D}
$$ 
for some $D>0$.
The function $l_0$ is twice continuously differentiable and there are constants $L_1'',L_2''$
such that
$$
L_1'' I\leq \nabla\nabla l_0\leq L_2'' I
$$
where $I$ is the $d\times d$ identity matrix.


\end{assumption}

\begin{remark}\label{foxes}{\rm
Assumption \ref{asp:H-special_form} implies that $\nabla l_0$ grows linearly, hence $l_0$ itself is locally Lipschitz with linearly growing Lipschitz-coefficient, that is,
$$
|l_0(\theta_1)-l_0(\theta_2)|\leq L_0(1+|\theta_1|+|\theta_2|)|\theta_1-\theta_2|,
$$
with some $L_0>0$, for all $\theta_1,\theta_2\in\mathbb{R}^d$.}
\end{remark}

In plain English, we consider $J$ which is smooth on a finite number of bounded domains (the interior
of the constraint sets $A_{i}(x)$, $i=1,\ldots, m_{s}$)
but may have discontinuities at the boundaries. Furthermore, $J$ (and hence also $U$) is required to be quadratic
``near infinity'' (on $A_{0}(x)$).

We briefly explain why such a hypothesis is not
restrictive for real-life applications. Normally,
there is a compact set $Q$ (e.g.\ a
cube or a ball) such that only
parameters from $Q$ are relevant, i.e.\ $U$ is
defined only on $Q$. Assume it has some stochastic
representation \begin{equation}\label{csicsi}
U(\theta)=\E[J(\theta,X_0)],\ \theta\in Q
\end{equation}
and a unique maximum $\theta^*\in Q$. 
Assume that $Q\subset B_D$ for some $D$.
Extend $U$ outside $B_D$ as $U(\theta)=-A|\theta|^2+B$
for suitable $A,B$. Extend $U$ and $J$ to $B_D\setminus Q$ as well
in such a way that $U$ is continuously 
differentiable, $U(\theta)<U(\theta^*)$ 
for all $\theta\neq \theta^*$ (see Section 4 of \cite{chau2019stochastic} for a
rigorous construction of this kind).
Set $J:=U$ outside $Q$.
Then our maximization procedure can be applied to this setting for finding
$\theta^*$. 

Defining $U=l_0$ (essentially) quadratic outside
a compact set is one way of solving the problem
that such procedures often leave their effective
domain $Q$. Other
solutions are resetting, see 
e.g.\ \cite{gerencser1992rate}; or an analysis
of the probability of divergence, see 
e.g.\ \cite{bmp}.

The next assumption postulates that the process
$X$ should be bounded and the conditional laws of $X_{k+1}$ should be absolutely
continuous with a bounded density.

\begin{assumption}\label{smoothness}
For each $k\in\mathbb{N}$, 
$$
P_k(X_{k+1}\in A)(\omega)=\int_{A}p_{k}(u,\omega)\, du,P\mbox{-a.s.,} A\in\mathcal{B}(\mathbb{R}^{d}),
$$
for some measurable $p_{k}:\mathbb{R}^{d}\times\Omega\to\mathbb{R}_{+}$ and 
there is a fixed constant $F$ such that $p_{k}(u,\omega)\leq F$ holds for all $k$, $\omega$, $u$. The random variable $X_0$
satisfies $|X_0|\leq K_0$ for some constant $K_0$.
\end{assumption}

Note that, by strong stationarity, the process $X_k$
is uniformly bounded under Assumption \ref{smoothness}.

We will assume a certain mixing property about the
process $X_n$ which we recall now. A family of $\mathbb{R}^d$-valued random
variables $Z_{i}$, $i\in  \mathcal{I}$
is called \emph{$L^r$-bounded} for some $r\geq 1$ if
$\sup_{i\in \mathcal{I}}\E|Z_i|^r<\infty$, here $\mathcal{I}$
may be an arbitrary index set.

For a random field $Y_n(\theta)$, $n\in \N$, $\theta\in\mathbb{R}^d$
bounded in $L^r$ for some $r\geq 1$, we
define, for all $n\in\mathbb{N}$,
\begin{align*}
    &M_r^n(Y)=\esssup_{\theta}\sup_{k\in\N}\E^{1/r}[\left\vert Y_{n+k}(\theta)\right\vert^r\vert\mathcal{F}_n],\\
    &\gamma_r^n(\tau,Y)=\esssup_{\theta}\sup_{k\geq\tau}
    \E^{1/r}\left[\left\vert Y_{n+k}(\theta) -\E[Y_{n+k}(\theta)\vert\mathcal{F}_{n+k-\tau}^{+}\vee\mathcal{F}_n]\right\vert^r\vert\mathcal{F}_n \right],\tau\geq 0,\\
&\Gamma_r^n(Y)=\sum_{\tau=0}^{\infty}\gamma_r^n(\tau,Y).
\end{align*}
These quantities clearly make sense also for any $L^r$-bounded
stochastic process $Y_n$,
$n\in\mathbb{N}$ (the essential suprema disappear
in this case). $M^n_r(Y)$ measures the (conditional) moments of $Y$ while $\Gamma_r^n(Y)$
describes its dependence structure (like covariance decay).
In particular, one can define
$M^n_r(X)$, $\Gamma^n_r(X)$. We clearly have $M^n_r(X)\leq K_0$ under Assumption \ref{smoothness}.
The quantities $\Gamma^n_r(X)$ will figure
in certain estimates later.


\begin{assumption}\label{asp:mix}
For some $\epsilon>0$, $\gamma_3^n(\tau,X)=O((1+\tau)^{-4-\epsilon})$, where the constant of $O(\cdot)${}
is independent of $\omega$, $\tau$ and $n$.
Furthermore,
$$
\E\left[\left\vert X_{n+k} -\E[X_{n+k}\vert\mathcal{F}_{n}^{+}]
\right\vert\right]=O(k^{-2-\epsilon}),\ k\geq 1,
$$    
where the constantof $O(\cdot)$ is independent of $n,k$.
\end{assumption}

Both requirements in Assumption \ref{asp:mix} are about how the effect of the past on the present
decreases as we go back farther in time.

\begin{example}
{\rm Let $\varepsilon_n$, $n\in\mathbb{N}$ be 
a bounded i.i.d.\ sequence in $\mathbb{R}^m$ with bounded density w.r.t.\ the Lebesgue measure and choose $\mathcal{F}_k:=\sigma(\varepsilon_j,\ j\leq k)$
and $\mathcal{F}_k^+:=\sigma(\varepsilon_j,\ j\geq k+1)$ .
Then $X_n:=\varepsilon_n$, $n\in\mathbb{N}$
satisfies Assumptions
\ref{smoothness} and \ref{asp:mix}. 
A causal infinite moving average process whose coefficients decay sufficiently fast is another pertinent example.
Indeed, using the argument of Lemma 4.2 of \cite{chau2016fixed} one can show that
$X_n:=\sum_{j=0}^{\infty}s_j\varepsilon_{n-j}$, $n\in\mathbb{N}$ satisfies Assumption \ref{asp:mix}
where the $\varepsilon_i$ are as above,
$s_0\neq 0$ and $|s_j|\leq (1+j)^{-\beta}$ holds for some $\beta>9/2$. 
Assumption \ref{smoothness} is also 
clearly satisfied in that model.}
\end{example}


\begin{remark}{\rm
A random field $Y_n(\theta), n\in \N$ is called uniformly conditionally $L$-mixing 
if $Y_n(\theta)$ is adapted to the filtration $\mathcal{F}_n$, $n\in\mathbb{N}$ for all $\theta$,
and the sequences $M_r^n(Y)$, $\Gamma_r^n(Y)$, $n\in\mathbb{N}$ are bounded in $L^r$ for each $r\geq 1$.
Our Assumption \ref{asp:mix} thus
requires a sort of related mixing property.
Conditional $L$-mixing was introduced in \cite{chau2016fixed}, inspired by \cite{gerencser1989class}. 
}
\end{remark}



\begin{subsection}{Decreasing gain stochastic approximation}
The usual assumption on the sequences $(\lambda_k)_{k=1,2,\dots}$ and $(c_k)_{k=1,2,\dots}$ in the definition of the recursive scheme (\ref{recursive_scheme}) 
are the following, see \cite{kiefer1952stochastic}:
\begin{eqnarray} \nonumber
    c_k &\rightarrow& 0,\ k\rightarrow\infty,\\ 
\nonumber    \sum_{k=1}^{\infty} \lambda_k &=& \infty,\\ 
\nonumber    \sum_{k=1}^{\infty} \lambda_kc_k &<& \infty, \\
\label{akck}    \sum_{k=1}^{\infty} \lambda_k^2 c_k^{-2} &<& \infty.
\end{eqnarray}

In the sequel we stick to a more concrete choice
which clearly fulfills the conditions in \eqref{akck} above.
\begin{assumption}\label{ac}
We fix $\lambda_0,c_0>0$, $\gamma\in (0,1/3)$ and set
$$
\lambda_k=\lambda_0 \int_{k}^{k+1}\frac{1}{u}\, du,
$$ 
and $c_k=c_0 k^{-\gamma}$, $k\geq 1$. 
We also assume $c_0\leq 1$.
\end{assumption}
Asymptotically $\lambda_{k}$ behaves like $\lambda_{0}/k$. However, our choice somewhat simplifies the otherwise
already involved theoretical analysis.

The ordinary differential equation associated with the problem is 
 \begin{align}
 \Dot{y}_t=\frac{\lambda_0}{t}G(y_t).\label{ode_dec_gain}   
 \end{align}
 
The idea to use an associated deterministic ODE to study the asymptotic properties of recursive schemes was introduced by Ljung in \cite{ljung1977}. The intuition behind this association is that on the long run the noise effects average out and the asymptotic behavior is determined by this 'mean' differential equation. Heuristic connection between the dynamics of the recursive scheme and the ODE can be seen if one looks at the Euler-discretization of the latter.

The solution of (\ref{ode_dec_gain}) with initial condition $y_s=\xi$ will be denoted by $y(t,s,\xi)$ for $0< s\leq t$.

\begin{assumption} \label{asp:stability}
    The ODE \eqref{ode_dec_gain} fulfills the stability assumption formulated below: 
    there exist $C^{*}>0$ and $\alpha>0$ such that 
    \begin{align*}
        \left\vert\frac{\partial y(t,s,\xi)}{\partial\xi}\right\vert\leq C^{*} \left(\frac{s}{t}\right)^{\alpha}
    \end{align*}
    for all $0<s<t$.
\end{assumption}

Our main result comes next.
 
\begin{theorem} \label{thm:convergencrate}
Let Assumptions \ref{asp:G}, \ref{asp:H-special_form}, \ref{smoothness}, \ref{asp:mix}, \ref{ac} 
and \ref{asp:stability} hold.
Then
$\E|\theta_n-\theta^{*}|=O(n^{-\chi}+n^{-\alpha})$, $n\geq 1$,
where
$\chi =\min\{\frac{1}{2} - \frac{3}{2}\gamma,\gamma\}$
and the constant in $O(\cdot)$ depends only on $\theta_0$.
\end{theorem}

To get the best result set $\gamma=\frac{1}{5}$. In this case the convergence rate is $\chi=\frac{1}{5}$
(provided that $\alpha\geq 1/5$).
For Kiefer-Wolfowitz procedures \cite{sacks} establishes 
a convergence rate $n^{-1/3}$ under fairly restrictive conditions (e.g.\ $J$ is assumed smooth and $X$ is i.i.d.).
Our approach is entirely different from that of \cite{sacks} and relies on the ODE method (see e.g.\ \cite{kushner})
in the spirit of \cite{gerencser1992rate,laci_spsa1,laci_spsa2} where so-called 
SPSA procedures were analysed.

Theoretical analysis in the present case is much more
involved for two reasons: the discontinuities of $J$ and the state-dependent setting (hardly analysed in
the literature at all). Our results are closest to \cite{laci_spsa2} where a rate of $n^{-2/7}$ is 
obtained for the SPSA algorithm (a close relative of Kiefer-Wolfowitz) imposing strong
smoothness assumptions on $J$. As already remarked, in the absence of smoothness ours is the first
study providing a convergence rate.
Eventual strengthening of our result seems to be difficult and will be object of
further investigations. 

We also point to two related papers where stochastic gradient procedures are analysed:
\cite{fort} treats the Markovian case while \cite{chau2016fixed} is about possibly non-Markovian
settings. In these studies, the gradient of $J$ is assumed to exist but it may be discontinuous.
Some of the ideas of \cite{chau2016fixed} apply in the present, more difficult case where even
$J$ fails continuity.

\end{subsection}

\end{section}

\begin{subsection}{Fixed gain stochastic approximation}

Let us also consider a modified recursive scheme 
\begin{align} \label{fixed_gain_scheme}
    \theta_{k+1}=\theta_k+aH(\theta_k,X_{k+1},c),\ k\in\mathbb{N},
\end{align} where $a$ and $c$ are fixed (small) positive reals, independent of $k$. In contrast with the previous scheme (\ref{recursive_scheme}), which is meant to converge to the maximum of the function, this method is expected to \emph{track} the maximum.

The ordinary differential equations associated with the problem are 
 \begin{align}
 \Dot{y}_t=\lambda G(y_t),\label{ode_fixed_gain}   
 \end{align}
for each $\lambda>0$.

Note that, by an exponential time change, one can show that Assumption \ref{asp:stability} 
on the ODE (\ref{ode_dec_gain}) implies (\ref{ode_fixed_gain}) being exponentially stable, i.e.\ satisfying 
\begin{align*}
        \left\vert\frac{\partial y(t,s,\xi)}{\partial\xi}\right\vert\leq C^{*} e^{-\alpha \lambda(t-s)},\ 0< s\leq t
    \end{align*}
for some $\alpha>0$ (possibly different from the one in (\ref{ode_dec_gain})).

\begin{theorem} \label{thm:fixed_gain}
Let Assumptions \ref{asp:G}, \ref{asp:H-special_form}, \ref{smoothness}, \ref{asp:mix} and 
\ref{asp:stability} hold. Then $\E|\theta_n-\theta^{*}|=O\left(\max\left(c^2,\sqrt{\frac{a}{c}}\right)+
    e^{-a\alpha n}\right)$ holds for all $n\geq 1$ where the constant in $O(\cdot)$ depends
    only on $\theta_{0}$.
\end{theorem}

Note that, similarly to the decreasing gain setting, this leads to the best choice being $c=a^{\frac{1}{5}}$.
We know of no other papers where the fixed gain case has been treated. In the case of stochastic
gradients there are many such studies obtaining a rate of $\sqrt{a}$ for
step size $a$, see e.g.\ \cite{chau2016fixed} and the references therein.
\end{subsection}

\end{section}

\begin{section}{Proofs}\label{prof}

The following lemma will play a pivotal role
in our estimates: it establishes the
\emph{conditional} Lipschitz-continuity of
the difference function obtained from $J$.

\begin{lemma}\label{conli} 
Under Assumptions \ref{asp:H-special_form}
and \ref{smoothness},
there is $C_{\flat}>0$ such that, for each $i=1,\ldots,d$ and $c\leq 1$,
$$
\E_{k}|J(\bar{\theta}_{1}+c\mathbf{e}_{i},X_{k+1})
-J(\bar{\theta}_{1}-c\mathbf{e}_{i},X_{k+1})-J(\bar{\theta}_{2}+c\mathbf{e}_{i},X_{k+1})
+J(\bar{\theta}_{2}-c\mathbf{e}_{i},X_{k+1})|\leq C_{\flat}[|\bar{\theta}_{1}-\bar{\theta}_{2}|+c^{2}]
$$
holds for all $k\in\mathbb{N}$ and for all pairs of $\mathcal{F}_{k}$-measurable $\mathbb{R}^{d}$-valued random variables
$\bar{\theta}_{1}$, $\bar{\theta}_{2}$.
\end{lemma}
\begin{proof}
{}
We assume that $m_{s}=1$, $m_{p}=1$, $m_p'=0$. We will shortly refer to the general case later.
We thus assume that $J(\theta,x)=l_{1}(\theta,x)1_{\{x\leq g(\theta)\}}+l_{0}(\theta)1_{A_{0}(x)}$ 
with some Lipschitz-continuous $g,l_{1}$ with Lipschitz-constant $L_{1}$ (for both). Let $K_{1}$ be an upper
bound for $l_{1}$ in $B_{D+2}$. 

{}
Consider first the event $A_{1}:=\{\bar{\theta}_{1},\bar{\theta}_{2}\in B_{D+1}\}$ and
the corresponding indicator $I_1:=1_{A_1}$. Note that on $I_1$ we have
$\bar{\theta}_j\pm c\mathbf{e}_i\in B_{D+2}$, $j=1,2$.
Now estimate 
\begin{eqnarray}\nonumber
& & \E_{k}|I_{1}l_1(\bar{\theta}_{1}+c\mathbf{e}_{i},X_{k+1})1_{\{X_{k+1}\leq 
g(\bar{\theta}_{1}+c\mathbf{e}_{i})\}}
-I_{1}l_1(\bar{\theta}_{2}+c\mathbf{e}_{i},X_{k+1})1_{\{X_{k+1}\leq 
g(\bar{\theta}_{2}+c\mathbf{e}_{i})\}}|\\
\nonumber &\leq& \E_{k} |I_{1} l_{1}(\bar{\theta}_{1}+c\mathbf{e}_{i},X_{k+1})1_{\{X_{k+1}\leq 
g(\bar{\theta}_{1}+c\mathbf{e}_{i})\}}
-I_{1}l_{1}(\bar{\theta}_{2}+c\mathbf{e}_{i},X_{k+1})1_{\{X_{k+1}\leq g(\bar{\theta}_{1}+c\mathbf{e}_{i})\}}|\\
\nonumber &+& \E_{k}|I_{1}l_{1}(\bar{\theta}_{2}+c\mathbf{e}_{i},X_{k+1})1_{\{
X_{k+1}\leq g(\bar{\theta}_{1}+c\mathbf{e}_{i})\}}
-I_{1} l_{1}(\bar{\theta}_{2}+c\mathbf{e}_{i},X_{k+1})1_{\{X_{k+1}\leq g(\bar{\theta}_{2}+c\mathbf{e}_{i})}\}|\\
\nonumber &\leq& L_{1} \E_{k}|\bar{\theta}_{1}-\bar{\theta}_{2}|+
K_{1} \sum_{j=1}^{m} 
\left[P_{k}(g^{j}(\bar{\theta}_{2}+c\mathbf{e}_{i})<X_{k+1}^{j}\leq g^{j}(\bar{\theta}_{1}+c\mathbf{e}_{i}))
+P_{k}(g^{j}(\bar{\theta}_{1}+c\mathbf{e}_{i})< X^{j}_{k+1}\leq g^{j}(\bar{\theta}_{2}+
c\mathbf{e}_{i}))\right]\\
&\leq& L_{1}|\bar{\theta}_{1}-\bar{\theta}_{2}|+
2mK_{1}L_{1}F|\bar{\theta}_{1}-\bar{\theta}_{2}|.
\label{ketcsicsi}
\end{eqnarray}
In the same way, we also get
\begin{eqnarray*}
\E_{k}|I_{1}l_1(\bar{\theta}_{1}-c\mathbf{e}_{i},X_{k+1})-I_{1}l_1(\bar{\theta}_{2}-c\mathbf{e}_{i},X_{k+1})|
&\leq& L_{1}|\bar{\theta}_{1}-\bar{\theta}_{2}|+2mK_{1}L_{1}F |\bar{\theta}_{1}-\bar{\theta}_{2}|.
\end{eqnarray*}
As $l_0$ is clearly Lipschitz on $B_{D+2}$,
we also have $|I_{1}l_0(\bar{\theta}_{1}\pm c\mathbf{e}_{i},X_{k+1})-I_{1}l_0(\bar{\theta}_{2}\pm c\mathbf{e}_{i},X_{k+1})|=O(|\bar{\theta}_{1}-\bar{\theta}_{2}|)$.

Let $L_2''$ be an upper bound for the second derivative $\nabla\nabla l_{0}$, recall Assumption 
\ref{asp:H-special_form}.
Now let $A_{2}$ be the event that the line from $\bar{\theta}_{1}$ to $\bar{\theta}_{2}$ does not
intersect $B_{D+1}$ at all, let $I_2:=1_{A_2}$. It follows
in particular that neither
$\bar{\theta}_1\pm c\mathbf{e}_i$ nor
$\bar{\theta}_2\pm c\mathbf{e}_i$ fall into 
$B_D$.
Since $J=l_{0}$ outside $B_{D}$ we can write, by the Lagrange mean value theorem,
\begin{eqnarray*}
& & \E_{k}I_{2}|J(\bar{\theta}_{1}+c\mathbf{e}_{i},X_{k+1})-J(\bar{\theta}_{2}+c\mathbf{e}_{i},X_{k+1})-
J(\bar{\theta}_{1}-c\mathbf{e}_{i},X_{k+1})
+J(\bar{\theta}_{2}-c\mathbf{e}_{i},X_{k+1})|\\
&=& 2c\E_{k}I_{2}|\partial_{\theta_i}l_{0}(\xi_{1})-\partial_{\theta_i}l_{0}(\xi_{2})|\\
&\leq& 2c\sup_{u\in\mathbb{R}^{d}}|\nabla
(\partial_{\theta_i}l_{0}(u))|\, \E_{k}|\xi_{1}-\xi_{2}|\\
&\leq& 2cL_2'' \E_{k}|\xi_{1}-\xi_{2}| \\
&\leq& 2cL_2'' [|\bar{\theta}_{1}-\bar{\theta}_{2}|+ 2c] \leq 2L_2''|\bar{\theta}_1-\bar{\theta}_2|+
4c^2 L_2''
\end{eqnarray*}
holds with some random variables $\xi_{j}\in [\bar{\theta}_{j}-c\mathbf{e}_{i},\bar{\theta}_{j}+c\mathbf{e}_{i}]$,
$j=1,2$, remembering our assumptions on $l_{0}$ and $c\leq 1$.

Turning to the event $\Omega\setminus (A_1\cup A_2)$, let us consider the directed straight line from $\bar{\theta}_{1}(\omega)$ to $\bar{\theta}_{2}(\omega)$ and let 
its first intersection point
with the boundary of $B_{D+1}$ be denoted by $\kappa_{1}(\omega)$ and its second intersection point by $\kappa_{2}(\omega)$.
In the case where there is only one intersection point it is denoted by $\kappa_{1}(\omega)$.
Let $I_{3}$ be the indicator of the event that there is only one intersection point ($\kappa_{1}$) with $B_{D+1}$
and that $\bar{\theta}_{1}$ is inside $B_{D+1}$.
The arguments of the previous two cases guarantee that
\begin{eqnarray*}
& & \E_{k}I_{3}|J(\bar{\theta}_{1}+c\mathbf{e}_{i},X_{k+1})-J(\bar{\theta}_{2}+c\mathbf{e}_{i},X_{k+1})-
J(\bar{\theta}_{1}-c\mathbf{e}_{i},X_{k+1})+J(\bar{\theta}_{2}-c\mathbf{e}_{i},X_{k+1})|\\
&\leq & \E_{k}I_{3}|J(\bar{\theta}_{1}+c\mathbf{e}_{i},X_{k+1})-J(\kappa_{1}+c\mathbf{e}_{i},X_{k+1})-
J(\bar{\theta}_{1}-c\mathbf{e}_{i},X_{k+1})+J(\kappa_{1}-c\mathbf{e}_{i},X_{k+1})|\\
&+& \E_{k}I_{3}|J(\kappa_{1}+c\mathbf{e}_{i},X_{k+1})-J(\bar{\theta}_{2}+c\mathbf{e}_{i},X_{k+1})-
J(\kappa_{1}-c\mathbf{e}_{i},X_{k+1})+J(\bar{\theta}_{2}-c\mathbf{e}_{i},X_{k+1})|\\
&=& O(|\bar{\theta}_1-\kappa_1|)+O(|\kappa_1-
\bar{\theta}_2|)+O(c^2)\\
&=& O(|\bar{\theta}_1-
\bar{\theta}_2|)+O(c^2).
\end{eqnarray*}


Similarly, if $I_{4}$ is the indicator of the event where there is one intersection point
and $\bar{\theta}_{2}$ is inside $B_{D+1}$ then we also get
\begin{eqnarray*}
& & \E_{k}I_{4}|J(\bar{\theta}_{1}+c\mathbf{e}_i,X_{k+1})-J(\bar{\theta}_{2}+c\mathbf{e}_i,X_{k+1})-
J(\bar{\theta}_{1}-c\mathbf{e}_i,X_{k+1})+J(\bar{\theta}_{2}-c\mathbf{e}_i,X_{k+1})|=O(|\bar{\theta}_1-
\bar{\theta}_2|+c^2).
\end{eqnarray*}

Let $I_{5}$ denote the indicator of the 
case where both $\bar{\theta}_{1}$, $\bar{\theta}_{2}$ are
outside $B_{D+1}$ and there are two intersection points $\kappa_{1}$, $\kappa_{2}$.
We get, as above,
\begin{eqnarray*}
& & \E_{k}I_{5}|J(\bar{\theta}_{1}+c\mathbf{e}_i,X_{k+1})-J(\bar{\theta}_{2}+c\mathbf{e}_i,X_{k+1})-
J(\bar{\theta}_{1}-c\mathbf{e}_i,X_{k+1})+J(\bar{\theta}_{2}-c\mathbf{e}_i,X_{k+1})|\\
&=& O(|\bar{\theta}_1-\kappa_1|)+
O(|\kappa_1-
\kappa_2|)+O(|\kappa_2-
\bar{\theta}_2|)+O(c^2)\\
&=& O(|\bar{\theta}_1-
\bar{\theta}_2|+c^2).
\end{eqnarray*}
Finally, in the remaining case
(where there is only one intersection point
with $B_{D+1}$ though both 
$\bar{\theta}_{1}$, $\bar{\theta}_{2}$ are
outside $B_{D+1}$) we similarly get
an estimate of the order
$O(|\bar{\theta}_1-
\bar{\theta}_2|+c^2)$ and hence
we eventually obtain the statement
of the lemma.


When $m_p=0$ and $m_p'=1$, the same ideas work.
When $m_p+m_p'>1$ we can rely on the following
elementary observation:
$$\left| \prod_{j=1}^{m_p} \mathbb{1}_{\{X_{k+1}\leq g^j(\bar{\theta}_1+c\mathbf{e}_i)\}}-  
\prod_{j=1}^{m_p} \mathbb{1}_{\{X_{k+1}\leq g^j(\bar{\theta}_2+
c\mathbf{e}_i)\}}  \right| \leq  
\sum_{j=1}^{m_p} \left| \mathbb{1}_{\{X_{k+1}\leq g^j(\bar{\theta}_1+c\mathbf{e}_i)\}}-  \mathbb{1}_{\{X_{k+1}\leq 
g^j(\bar{\theta}_2+c\mathbf{e}_i)\}}  \right|,{}
$$
and its counterpart for the $h^j$.
Estimates can be repeated for each summand in the
definition of $J$ so the case $m_s>1$ follows, too.
\end{proof}

The arguments of the previous lemma, \eqref{ketcsicsi} in particular, also
give us the following:

\begin{lemma}\label{colis} 
Under Assumptions \ref{asp:H-special_form} and
\ref{smoothness}, there is $C_{\natural}>0$ such that, for each $i=1,\ldots,d$,
$$
\E_{k}|J(\bar{\theta}+c\mathbf{e}_{i},X_{k+1})
-J(\bar{\theta}-c\mathbf{e}_{i},X_{k+1})|
\leq C_{\natural}c,\quad 0<c\leq 1,
$$
holds for all $k\in\mathbb{N}$ and for all $\mathcal{F}_{k}$-measurable $B_{D+1}$-valued random variables $\bar{\theta}$.\hfill $\Box$
\end{lemma}

\begin{subsection}{Moment estimates}
 
 
In this subsection, we will prove that the first moments of our iteration scheme remain bounded.
This will be followed by other moment estimates. We start with a preliminary lemma on deterministic sequences.

\begin{lemma}\label{sequence}
Let $x_{k}\geq 0$, $k\in\mathbb{N}$ be a sequence, let $\zeta_{k}>0$, $k\geq 1$
be another sequence. If they satisfy $\nu\zeta_k<1$, $k\geq 1$ and
$$
x_{k}\leq (1-\nu\zeta_{k})x_{k-1}+\underline{c}\zeta_{k},\ k\geq 1,
$$	
with some $\underline{c},\nu>0$ then
$$
\sup_{k\in\mathbb{N}}x_{k}\leq x_{0}+ \frac{\underline{c}}{\nu}. 
$$
\end{lemma}	
\begin{proof}
Following the argument of Lemma 1 in \cite{eric}, 
we notice that
$$
x_{k}\leq \prod_{i=1}^k (1-\nu\zeta_{i}) x_{0} + \underline{c}
\sum_{i=1}^{k} \zeta_{i}\prod_{j=i+1}^k (1-\nu\zeta_{j}),
$$
where an empty product is meant to be $1$. 
We can write
\begin{eqnarray*}
& & \sum_{i=1}^{k} \zeta_{i}\prod_{j=i+1}^k (1-\nu\zeta_{j})\\
&=& \frac{1}{\nu}\sum_{i=1}^{k}\left({}
\prod_{j=i+1}^{k}(1-\nu \zeta_{j})-\prod_{j=i}^{k}(1-\nu \zeta_{j})\right){}
\\
&\leq& \frac{1}{\nu}.
\end{eqnarray*}	
This shows the claim.
\end{proof}


Certain calculations are easier to carry out if we consider the continuous-time embedding of the discrete time 
processes. Consider the 
following 
extension ${\theta}_t$,
$t\in\mathbb{R}_+$ of $\theta_k$, $k\in\mathbb{N}$: let 
$$
{\theta}_t:=\theta_{k}+\int_{k}^{t}a_{u}H(u,\theta_{k})\, du
$$
for all $k\in\mathbb{N}$ and for all $k\leq t<k+1$, where
$H(u,\theta)=H(\theta,X_{k+1},c_k)$  for all 
$k\in\mathbb{N}$
and for all $k\leq u<k+1$, $c_u=c_k$ and $a_{u}=\lambda_{0}/\max\{u,1\}$, $u\geq 0$. Extend the filtration
to continuous time by $\mathcal{F}_t:=\mathcal{F}_{\lceil t\rceil}$,
$t\in\mathbb{R}_+$.
Now fix $\mu>1$. We introduce an auxiliary
process that will play a crucial role
in later estimates. For each $n\geq 1$
and for $\lceil n^{\mu}\rceil\leq t<\lceil (n+1)^{\mu}\rceil$ define 
$\overline{y}_t:=y(t,\lceil n^{\mu}\rceil,\theta_{\lceil n^{\mu}\rceil})$, 
 i.e. the solution of \eqref{ode_dec_gain} starting at $\lceil n^{\mu}\rceil$ with initial condition 
 $\overline{y}_{\lceil n^{\mu}\rceil}=\theta_{\lceil n^{\mu}\rceil}$. 
 
We introduce the $L^1$-norm
$$
\Vert Z\Vert_1:=\E|Z|,
$$
for each $\mathbb{R}^d$-valued
random variable $Z$.

\begin{lemma}\label{momentus} 
Under Assumptions \ref{asp:H-special_form} and
\ref{smoothness}, we have
$$
\sup_{t\geq 1}\|\overline{y}_{t}\|_{1}+
\sup_{t\geq 1}
E\|\theta_{t}\|_{1}<\infty.
$$
\end{lemma}
\begin{proof}
Note that $2c_{k}H^{j}(\theta,x,c_{k})=
l_{0}(\theta+c_{k}\mathbf{e}_{j})-l_{0}(\theta-c_{k}\mathbf{e}_{j})$, for all $x$, $j=1,\ldots,d$  when $\theta\notin B_{D+1}$. 
Furthermore, the function $l_{0}(\theta+c_{k}\mathbf{e}_{j})-l_{0}(\theta-c_{k}\mathbf{e}_{j})$ is
Lipschitz on $B_{D+1}$ which, together with
Lemma \ref{colis} implies
\begin{equation}\label{subdo}
\left\Vert\frac{l_{0}(\theta+c_{k}\mathbf{e}_{j})-l_{0}(\theta-c_{k}\mathbf{e}_{j})}{2c_{k}}-H^{j}(\theta,X_{k+1},c_{k})\right\Vert_1\leq \bar{C},\ \theta\in\mathbb{R}^d,
\end{equation}
for a fixed constant $\bar{C}$.
Clearly,
\begin{eqnarray*}
||\theta_{k+1}||_{1} &\leq&  \left\|\theta_{k}-\frac{\lambda_{k}}{2c_{k}} \sum_{j=1}^{d}
[l_{0}(\theta_{k}+c_{k}\mathbf{e}_{j})-l_{0}(\theta_{k}-c_{k}\mathbf{e}_{j})]\mathbf{e}_{j}\right\|_{1}\\
&+& 
\left\|\frac{\lambda_{k}}{2c_{k}} \sum_{j=1}^{d}
[l_{0}(\theta_{k}+c_{k}\mathbf{e}_{j})-l_{0}(\theta_{k}-c_{k}\mathbf{e}_{j})]\mathbf{e}_{j}-
\lambda_{k}H(\theta_{k},X_{k+1},c_{k})\right\|_{1}.
\end{eqnarray*}
Note that, by Assumption \ref{asp:H-special_form},
$l_0$ is strongly convex, in particular,
$$
\langle \nabla l_0(\theta)-\nabla l_0(0),\theta\rangle\geq A_0|\theta|^2,\ \theta\in\mathbb{R}^d
$$
for all $\theta$, with some $A_0>0$. Hence also
$$
\langle \nabla l_0(\theta),\theta\rangle\geq A|\theta|^2-B,
\ \theta\in\mathbb{R}^d
$$
for suitable $A,B>0$. But then for all $a>0$
small enough,
$$
|\theta-a\nabla l_0(\theta)|\leq (1-A'a)|\theta|+a B',\ \theta\in\mathbb{R}^d
$$
for suitable $A',B'>0$.
By the mean value theorem, 
$$
l_{0}(\theta_{k}+c_{k}\mathbf{e}_{j})-
l_{0}(\theta_{k}-c_{k}\mathbf{e}_{j})=
2c_k \partial_j l_0(\xi_j)
$$
for some  random variable $\xi_j\in 
[\theta_{k}-c_{k}\mathbf{e}_{j},\theta_{k}+c_{k}\mathbf{e}_{j}]$.
Since $\nabla l_0$ is Lipschitz,
$$
\max_j\left\|\nabla l_0(\theta_k)-\nabla l_0(\xi_j)\right\|_{1}\leq L'
$$
for some $L'>0$.
It follows then easily that, for $k\geq k_0$ large enough such
that $\lambda_k$ is small enough, 
\begin{eqnarray*}
& & \left\|\theta_{k}-\frac{\lambda_{k}}{2c_{k}} \sum_{j=1}^{d}
[l_{0}(\theta_{k}+c_{k}\mathbf{e}_{j})-l_{0}(\theta_{k}-c_{k}\mathbf{e}_{j})]\mathbf{e}_{j}\right\|_{1}\\
&\leq& \lambda_k dL'+
||\theta_k-\lambda_k\nabla l_0(\theta_k)||_1\\
&\leq& \lambda_k (B'+dL')+(1-A' \lambda_k)||\theta_k||_1
\end{eqnarray*}
holds.
By \eqref{subdo},
\begin{eqnarray*}
\left\|\frac{\lambda_{k}}{2c_{k}} \sum_{j=1}^{d}
[l_{0}(\theta_{k}+c_{k}\mathbf{e}_{j})-l_{0}(\theta_{k}-c_{k}\mathbf{e}_{j})]\mathbf{e}_{j}-
\lambda_{k}H(\theta_{k},X_{k+1},c_{k})\right\|_{1}\leq 
\lambda_{k}d\bar{C}.
\end{eqnarray*}

Apply Lemma \ref{sequence} with the choice $x_{k}:=||\theta_{k}||_{1}$, 
$\underline{c}:=d (L'+\bar{C})+B'$ and $\zeta_{k}:=\lambda_{k}$, $\nu:=A'$ to obtain that $\sup_{k\geq k_0}\|\theta_{k}\|_{1}<\infty$.
Then trivially also $\sup_{n\in\mathbb{N}}\|\theta_{n}\|_{1}<\infty$ holds, which easily implies $\sup_{t\geq 1}\|\theta_{t}\|_{1}<\infty$ as well. 

Now turning to $\overline{y}_{t}$
we see that, for $n\geq 1$ and
$\lceil n^{\mu}\rceil\leq t<\lceil (n+1)^{\mu}\rceil$, 
\begin{eqnarray*}& & |\overline{y}_{t}-\theta^{*}|\\
&=& |\overline{y}_{t}-y(t,\lceil n^{\mu}\rceil,\theta^{*})|\\
&\leq& |\theta_{\lceil n^{\mu}\rceil}-\theta^{*}|
C^{*},
\end{eqnarray*}
finishing the proof.
\end{proof}



\begin{lemma} \label{lem:quad_growth} Let Assumptions \ref{asp:H-special_form}
and \ref{smoothness} hold. Then there exists $C_l>0$ such that $\sup_{k\geq 1}|J(\theta,X_k)|\leq 
C_l(1+\vert\theta\vert^2)$.
\end{lemma}
\begin{proof}
Recall that 
\begin{align*}
   \vert J(\theta,x)\vert  \leq |l_0(\theta)|+ \sum _{i=1}^{m_s} 1_{A_i(x)}\left \vert  l_i (x,\theta)  \right\vert,
\end{align*}
where the functions $l_i$ are
bounded on the bounded sets $\cup_{x\in\mathbb{R}^d} A_i(x)$ for $i=1,\ldots,d$, and $l_0$ grows
quadratically.
\end{proof}

The difficulty of the following lemma consists in
handling the discontinuities and the dependence of the sequence $X_{k}$ at the same time.

\begin{lemma} \label{lem:M-Gamma-estimate} Let Assumptions \ref{asp:H-special_form}
and \ref{smoothness} hold. Then for each $R>0$ the random field $J(\theta,X_n)$, $\theta\in B_R$, $n\in\mathbb{N}$ satisfies
\begin{align*}
    & M_3^n(J(\theta, X))\leq  C_l(1+R^2),\\
    & \Gamma_3^n(J(\theta, X))\leq L(1+R^{2}),
\end{align*}
for some $L>0$, where $C_l$ is as in Lemma  \ref{lem:quad_growth}.
\end{lemma} 
\begin{proof} The first statement is clear from Lemma \ref{lem:quad_growth}.
Let $n\geq 0$, $\tau\geq 1$ be fixed. For $k\geq\tau$, define 
$X_k^+ = \E[X_{n+k}\vert \mathcal{F}_{n+k-\tau}^+\vee\mathcal{F}_{n}]$.
For the sake of simplicity, assume that $m_s=1$ in the definition of $J$, $m_p=0$, but the same argument would 
work for several summands, too. We also take the process $X$ unidimensional ($m:=1$)
noting that the same arguments easily carry over to a general $m$. 

We now perform an auxiliary estimate. Let $\epsilon_{\tau}>0$ be a parameter
to be chosen later and let $1\leq j\leq m_{p}'$. We will write $h$ below instead of $h_{1}$. Define $Z_{k}=X_{n+k}-X_{k}^{+}$ and estimate
\begin{align*}
  &  \E_n \left|1_{\{X_{n+k}>h^j(\theta)\}}-1_{\{X_{k}^{+} > h^j(\theta\})}\right|^3 
  = \E_n \left|1_{\{X_{n+k}>h^j(\theta)\}}-1_{\{X_{k}^{+} > h^j(\theta\})}\right|
  \\ &\leq \p_{n} \left( X_{n+k}\in\left( h^j(\theta)-|Z_{k}|,h^j(\theta)+|Z_{k}| \right) \right) \\
  &\leq \p_{n} \left( X_{n+k}\in\left( h^j(\theta)-|Z_{k}|,h^j(\theta)+|Z_{k}| \right), 
  |Z_{k}|\leq \epsilon_{\tau} \right)+ \p_{n}(|Z_{k}|\geq\epsilon_{\tau})  \\
   &\leq 2F\epsilon_{\tau}+\frac{\E_{n}[|X_{n+k}-X_{k}^{+}|^3]}{\epsilon_{\tau}^3},\\
\end{align*}
where the last inequality follows from Assumption \ref{smoothness}  and the Markov inequality. Now
estimate
\begin{eqnarray*}
    & & \E_{n}^{1/3} \left|  \left( \prod_{j=1}^{m_p'}{1}_{\{X_{n+k}>h^j(\theta)\}}\right) 
    l_1 (X_{n+k},\theta) - \left( \prod_{j=1}^{m_p'} {1}_{\{X_{k}^{+}>h^j(\theta)\}}\right) l_1 
    (X_{k}^{+},\theta) \right|^3 \\
    &\leq& \E_{n}^{1/3} \left| \left( \prod_{j=1}^{m_p'} {1}_{\{X_{n+k}>h^j(\theta)\}}-  
    \prod_{j=1}^{m_p'} {1}_{\{X^{+}_{k}>h^j(\theta)\}}  \right) l_1 (X_{n+k},\theta)     \right|^3 \\
    &+&  \E_{n}^{1/3} \left|\left( l_1 (X_{n+k},\theta) - l_1 (X_{k}^{+},\theta)\right) 
    \prod_{j=1}^{m_p'} {1}_{\{X_{k}^{+}>h^j(\theta)\}}\right|^3 \\
&\leq& \E^{1/3}_{n}\left[
\left|\sum_{j=1}^{m_p'} 
   \left| {1}_{\{X_{n+k}>h^j(\theta)\}}-  {1}_{\{X^{+}_{k}>h^j(\theta)\}} 
   \right|\right|^3(L_1(|X_{n+k}|+R)+|l_1(0,0)|)^3
   \right]+L_1\E_{n}^{1/3}|X_{n+k}-X_{k}^+|^3\\
   &\leq& C_{1}(1+R)\left(\epsilon_{\tau}^{1/3}+\frac{\E_{n}^{1/3}[|X_{n+k}-X_{k}^{+}|^3]}{\epsilon_{\tau}}\right)
\end{eqnarray*}
for some $C_{1}$,
where we used the Lipschitz-continuity of the function $l_1$, as well as the observation that 
$$\left| \prod_{j=1}^{m_p'} {1}_{\{X_{n+k}>h^j(\theta)\}}-  
\prod_{j=1}^{m_p'} {1}_{\{X^{+}_{k}>h^j(\theta)\}}  \right| \leq  
\sum_{j=1}^{m_p'} \left| {1}_{\{X_{n+k}>h^j(\theta)\}}-  
{1}_{\{X^{+}_{k}>h^j(\theta)\}}  \right|.{}
$$
A similar estimate works for $l_0$ but
we get the upper bound
\begin{eqnarray*}
& & \E_{n}^{1/3}\left| 1_{A_{0}(X_{n+k})}l_{0}(\theta)-1_{A_{0}(X_{k}^{+})}l_{0}(\theta)\right|^{3} \\
&\leq&  C_{1}(1+R^{2})\left(\epsilon_{\tau}^{1/3}+\frac{\E_{n}^{1/3}[|X_{n+k}-X_{k}^{+}|^3]}{\epsilon_{\tau}}\right)
\end{eqnarray*}
instead.
For the second inequality of the present lemma, note first
that Lemma \ref{mall} below implies
\begin{align*}
    &\E^{1/3}_{n}\left[\left\vert J(\theta,X_{n+k}) -\E[J(\theta,X_{n+k})\vert\mathcal{F}_{n}\vee\mathcal{F}_{n+k-\tau}^{+}]\right\vert^3\right]\\
   \leq & 2\E_{n}^{1/3}\left[\left\vert J(\theta,X_{n+k}) -J(\theta,X_{k}^+)\right\vert^3\right], 
\end{align*}
hence it suffices to estimate the latter
quantity. From our previous estimates 
it follows that, for some $C>0$,
\begin{equation}\label{juhh}
\E_n^{1/3}\left[|J(\theta,X_{n+k})-J(\theta,X_k^+)|^3\right]\leq C(1+R^{2})\left[\sqrt[3]{\epsilon_{\tau}}
+\frac{\E^{1/3}_n|X_{n+k}-X_{k}^+|^{3}}{\epsilon_{\tau}}\right]
\end{equation}
Choose $\epsilon_{\tau}:=(\tau+1)^{-3-\epsilon/2}$.
Summing up the right-hand side for $\tau\geq 1$ we see that, by Assumption \ref{asp:mix}, 
the sum has an upper bound independent of $k$. The statement follows as the case $\tau=0${}
is easy.\end{proof}

\end{subsection}

\begin{subsection}{Decreasing gain case}

The following lemma contains the core estimates of the present paper.

\begin{lemma}\label{lem:dec_gain} Let $n\geq 1$. Let $\lceil n^{\mu}\rceil\leq t<
\lceil (n+1)^{\mu}\rceil$ for $\mu:=1/\gamma$ and let Assumptions \ref{asp:G}, \ref{asp:H-special_form}, \ref{smoothness}, \ref{asp:mix}, \ref{ac} 
and \ref{asp:stability} hold. Then
$\E|\theta_t-\overline{y}_t|=O(n^{-\beta })$, where $\beta =\min\left(\frac{1}{2\gamma}-\frac{1}{2}, 2\right)$.
\end{lemma}

\begin{proof}
For $\lceil n^{\mu}\rceil\leq t
< \lceil (n+1)^{\mu}\rceil$,
\begin{align*}
    \vert \theta_{\lfloor t\rfloor} - \overline{y}_t\vert &\leq{}  
    \vert \overline{y}_{\lfloor t\rfloor} - \overline{y}_t\vert+
    \vert \theta_{\lfloor t\rfloor} - 
    \overline{y}_{\lfloor t\rfloor}\vert \\
    &\leq     \int\limits_{\lfloor t\rfloor}^{t} a_u 
     \left|G(\overline{y}_u)\right| \,du
     +
\left| \int\limits_{\lceil n^{\mu}\rceil}^{\lfloor t\rfloor} a_u \left( H(u,\theta_{\lfloor u\rfloor})-
    G(\overline{y}_u) \right) \,du \right|\\
&\leq a_{n^\mu} \int\limits_{\lfloor t\rfloor}^{t} \left|G(\overline{y}_u)\right| \,du{}
\\
&+  \left| \int\limits_{\lceil n^{\mu}\rceil}^{\lfloor t\rfloor} a_u \left( H(u,\theta_{\lfloor u\rfloor})-H(u,\overline{y}_u) \right)\,du \right| \\
    &+ \left| \int\limits_{\lceil n^{\mu}\rceil}^{\lfloor t\rfloor} a_u \left( H(u,\overline{y}_u)-\E[H(u,\overline{y}_u)| \mathcal{F}_{\lceil n^{\mu}\rceil} ] \right) \,du\right| \\
    &+ \left| \int\limits_{\lceil n^{\mu}\rceil}^{\lfloor t\rfloor} a_u \left( \E[H(u,\overline{y}_u)| \mathcal{F}_{\lceil n^{\mu}\rceil} ]-G(\overline{y}_u) \right)\,du\right| \\
    &=:\Sigma_0+\Sigma_1+\Sigma_2+\Sigma_3.
\end{align*}

\noindent\textsl{Estimation of $\Sigma_0$.}
Since $G$ has at most linear growth, 
Lemma \ref{momentus} guarantees that
$$
\E[\Sigma_0]=O\left(a_{n^{\mu}}\int_{\lfloor t\rfloor}^{t}(\E|\overline{y}_{u}|+1)\, du\right)=O(n^{-\mu}).
$$

\noindent\textsl{Estimation of $\Sigma_1$.}
Recall that, by the tower property for
conditional expectations,
$$
\E \left| H(u,\theta_u)-H(u,\overline{y}_{u})\right|=
\E\E_{k} \left| H(u,\theta_u)-H(u,\overline{y}_{u})\right|
$$
for all $k\in\mathbb{N}$. Applying this observation
to $k=\lfloor u\rfloor$, Lemma \ref{conli}
implies that
\begin{align} \nonumber
    \E[\Sigma_1] & = \E \left|  \int\limits_{\lceil n^{\mu}\rceil}^{\lfloor t\rfloor} a_u \left( H(u,\theta_{\lfloor u\rfloor})-H(u,\overline{y}_{u}) \right)\,du\right|\\
    & \leq   \int\limits_{\lceil n^{\mu}\rceil}^{t} a_u \E \left| H(u,\theta_{\lfloor u\rfloor})-H(u,\overline{y}_{u})\right| \,du \nonumber
    \\
    & \leq C_{\flat}  \int\limits_{\lceil n^{\mu}\rceil}^{t} \frac{a_u}{c_u} \E
    \left|\theta_{\lfloor u\rfloor}-\overline{y}_{u}\right|\,du +C_{\flat}\int\limits_{\lceil n^{\mu}\rceil}^{t} \frac{a_{u}}{c_{u}}c_{u}^{2}\, du
\label{csillagocska}
\end{align} 
Henceforth we will denote
$$
\Sigma_1':=C_{\flat}\int\limits_{\lceil n^{\mu}\rceil}^{\lceil (n+1)^{\mu}\rceil} \frac{a_{u}}{c_{u}}c_{u}^{2}\, du.
$$
Notice that 
$$
E[\Sigma_{1}']=O(n^{-\mu\gamma-1})=O(n^{-2}).
$$

\noindent\textsl{Estimation of $\Sigma_2$.}
Notice that $H(u,\bar{\theta})=\E[H(u,\bar{\theta})| \mathcal{F}_{\lceil n^{\mu}\rceil} ]$ for all $\mathcal{F}_{\lceil n^{\mu}\rceil}$-measurable
$\bar{\theta}$ such that a.s.\ $\bar{\theta}\notin B_D$ since $J(\theta,x)$ does not depend on $x$ 
outside $B_{D}$ by Assumption \ref{asp:H-special_form}. Thus
\begin{align*}
   \Sigma_2 
   &\leq\sup_{\lceil n^{\mu}\rceil \leq t< \lceil (n+1)^{\mu}\rceil}\left| 
   \int\limits_{\lceil n^{\mu}\rceil}^{t} a_u \1_{\{\overline{y}_{u}\in B_{D}\}}\left( H(u,\overline{y}_u)-\E[H(u,\overline{y}_u)| \mathcal{F}_{n^{\mu}} ] \right)\,du\right|.
\end{align*}{}

We will use the inequality of Theorem \ref{elso} below with $r=3$, with
$\mathcal{R}_{t}:=\mathcal{F}_{t+\lceil n^{\mu}\rceil}$, $t\in\mathbb{R}_{+}$, 
$\mathcal{R}_t^+:=\mathcal{F}_{t+\lceil n^{\mu}\rceil}^+$ with
the process defined by 
\begin{equation}\label{defw}
W_t= 
\1_{\{\overline{y}_{t+\lceil n^{\mu}\rceil}\in B_{D}\}} c_{t+\lceil n^{\mu}\rceil}\left(H(t,{}
\overline{y}_{t+\lceil n^{\mu}\rceil})-\E[H(t,\overline{y}_{t+\lceil n^{\mu}\rceil})\vert \mathcal{F}_{\lceil n^{\mu}\rceil}]
\right),\ t\geq 0
\end{equation} and with the function $f_t={a_{t+\lceil n^{\mu}\rceil}}/{c_{t+\lceil n^{\mu}\rceil }}$. 
Note that $\{\overline{y}_{t}\in B_{D}\}\in\mathcal{F}_{\lceil n^{\mu}\rceil}$ for all 
$\lceil n^\mu\rceil\leq t<\lceil (n+1)^\mu\rceil$.
We get from Lemma \ref{wales} below and
from the cited inequality that
\begin{align*}
    \E[\Sigma_2] &=  \E[\E[\Sigma_2\vert \mathcal{F}_{\lceil n^{\mu}\rceil}]] \leq  
    \E[
   \E^{1/3}[\Sigma_2^3\vert \mathcal{F}_{\lceil n^{\mu}\rceil}]] \\
    &\leq C'(3) \left(\int\limits_{\lceil n^{\mu}\rceil}^{\lceil (n+1)^{\mu}\rceil}\left(\frac{a_u}{c_u}\right)^2\,du\right)^{1/2} 
    \E[\tilde{M}_3+\tilde{\Gamma}_3]\\
    &\leq C'(3) \left(\int\limits_{\lceil n^{\mu}\rceil}^{\lceil (n+1)^{\mu}\rceil}\left(\frac{a_u}{c_u}\right)^2\,du\right)^{1/2}C(1+D^2).
\end{align*}
We thus get
$$
\E[\Sigma_2] =O\left(\int\limits_{\lceil n^{\mu}\rceil}^{\lceil (n+1)^{\mu}\rceil}\left(\frac{a_u}{c_u}\right)^2\,du\right)^{1/2}
=O\left(n^{\frac{-\mu+2\mu\gamma-1}{2}}\right).$$

\noindent\textsl{Estimation of $\Sigma_3$.}
\begin{align}\nonumber
    \E[\Sigma_3 ] & \leq
    \E\left[\int\limits_{\lceil n^{\mu}\rceil}^{t} a_u 
    \left| \E[H(u,\overline{y}_{u})\vert \mathcal{F}_{\lceil n^{\mu}\rceil}]-G(\overline{y}_{u})\right|\,du \right]\\
    &\leq \nonumber \E \left[\int\limits_{\lceil n^{\mu}\rceil}^{t} a_u
    \sup_{\vartheta\in \mathbb{R}^{d}}\left\vert
    \E[H(\vartheta,X_{\lfloor u\rfloor+1},c_{\lfloor u\rfloor})\vert \mathcal{F}_{\lceil n^{\mu}\rceil}]-  \E[H(\vartheta,X_{\lfloor u\rfloor+1},c_{\lfloor u\rfloor}) ]\right\vert\,du \right] \\
    & + \E\left[\int\limits_{\lceil n^{\mu}\rceil}^{t} a_u 
    \sup_{\vartheta\in\mathbb{R}^{d}}\left\vert  \E[H(\vartheta,X_{\lfloor u\rfloor+1},c_{\lfloor u\rfloor}) ] - G(\vartheta)\right\vert\,du  \right] \label{mangalica}
\end{align}

To handle the second sum, note that, for each $i=1,\ldots,d$,{}
$$
\E[H^{i}(\vartheta,X_{k+1},c_k) ] = \frac{U(\vartheta + c_k\mathbf{e}_{i})-U(\vartheta - c_k\mathbf{e}_{i})}{2c_k}=
G^{i}(\xi^{i}_k)
$$ 
for some $\xi^{i}_{k}\in [\vartheta - c_k\mathbf{e}_{i},\vartheta + c_k\mathbf{e}_{i}]$.
The Lipschitz continuity of $G$ implies that $|G^{i}(\xi^{i}_k)- G^{i}(\vartheta)|\leq L_G c_k$
so
$$\E\left[\int\limits_{\lceil n^{\mu}\rceil}^{t} a_u  \sup_{\vartheta\in\mathbb{R}^d}\left\vert
\E[H(\vartheta,X_{\lfloor u\rfloor+1},
c_{\lfloor u\rfloor})] - G(\vartheta)\right\vert \, du\right]  \leq \int\limits_{\lceil n^{\mu}\rceil}^{t} a_u d L_G  c_{\lceil u\rceil}\, du 
=O\left(\int\limits_{\lceil n^\mu\rceil}^{\lceil (n+1)^\mu\rceil} u^{-1-\gamma}\, du
\right)=O(n^{-2}).$$

Now we turn to the first sum in \eqref{mangalica}. 
 Define $X_k^+ = \E[X_{k}\vert \mathcal{F}_{\lceil n^{\mu}\rceil}^+]$, $k\geq \lceil n^{\mu}\rceil$.
First let us estimate $$\E_{\lceil n^{\mu}\rceil}[\vert H(\vartheta,X_{k+1+\lceil n^{\mu}\rceil},c_{k+\lceil n^{\mu}\rceil}) - 
H(\vartheta,X_{k+1+\lceil n^{\mu}\rceil}^+,c_{k+\lceil n^{\mu}\rceil})\vert].$$
Fix $\epsilon_{k}>0$ to be chosen later.
By an argument similar to that of Lemma \ref{lem:M-Gamma-estimate} (using the first instead of the third moment
in Markov's inequality) we get that, for some constant $C_{1}$,
\begin{eqnarray*}
& & c_{k+\lceil n^{\mu}\rceil}\E_{\lceil n^{\mu}\rceil}[\vert 
H(\vartheta,X_{\lceil n^{\mu}\rceil+k+1},c_{k+\lceil n^{\mu}\rceil}) - 
H(\vartheta,X_{\lceil n^{\mu}\rceil+k+1}^+,c_{k+\lceil n^{\mu}\rceil})\vert]\\
&\leq& 
C_{1}\left[\epsilon_{k}+
\frac{\E_{\lceil n^{\mu}\rceil}[|X_{\lceil n^{\mu}\rceil+k+1}-X_{\lceil n^{\mu}\rceil+k+1}^{+}|]}{\epsilon_{k}}\right].
\end{eqnarray*}


Choose $\epsilon_k=(1+k)^{-1-\epsilon/2}$. Then using Assumption 
\ref{asp:mix} we get $$c_{k+\lceil n^{\mu}\rceil}
\sup_{\vartheta\in \mathbb{R}^{d}}
\E[|H(\vartheta,X_{\lceil n^{\mu}\rceil+k+1},c_{\lceil n^{\mu}\rceil+k})-
H(\vartheta,X_{\lceil n^{\mu}\rceil+k+1}^+,c_{k+\lceil n^{\mu}\rceil})|\vert 
\mathcal{F}_{\lceil n^{\mu}\rceil}] = 
O(k^{-1-\epsilon/2})$$
which also implies 
$$
c_{k+\lceil n^{\mu}\rceil}\sup_{\vartheta\in\mathbb{R}^{d}}
\E[|H(\vartheta,X_{\lceil n^{\mu}\rceil+k+1},c_{\lceil n^{\mu}\rceil+k})-
H(\vartheta,X_{\lceil n^{\mu}\rceil+k+1}^+,c_{\lceil n^{\mu}\rceil+k})|]= 
O(k^{-1-\epsilon/2}).
$$

Since 
$\E[H(\vartheta,X_{k+1}^+,c_{\lceil n^{\mu}\rceil+k})\vert 
\mathcal{F}_{\lceil n^{\mu}\rceil}]=\E[H(\vartheta,X_{k+1}^+,c_{k+\lceil n^{\mu}\rceil})]$ 
for $k\geq \lceil n^{\mu}\rceil$ by independence of 
$\mathcal{F}_{\lceil n^{\mu}\rceil}$ and $\mathcal{F}_{\lceil n^{\mu}\rceil}^{+}$, we have
\begin{align*}
& \left[\int\limits_{\lceil n^{\mu}\rceil}^{t} a_u \E\sup_{\vartheta\in \mathbb{R}^{d}}\left\vert  
    \E[H(\vartheta,X_{\lfloor u\rfloor+1},c_{\lfloor u\rfloor})\vert \mathcal{F}_{\lceil n^{\mu}\rceil}]-
    \E[H(\vartheta,X_{\lfloor u\rfloor+1},c_{\lfloor u\rfloor}) ]\right\vert \,du\right] 	\\
&\leq \int\limits_{\lceil n^{\mu}\rceil}^{\infty} 
\frac{a_u}{c_u} 
c_u\E\sup_{\vartheta\in \mathbb{R}^{d}}
\E[\left\vert H(\vartheta,X_{\lfloor u\rfloor+1},c_k)-H(\vartheta,X_{\lfloor u\rfloor+1}^+,c_k)\right\vert \vert 
\mathcal{F}_{\lceil n^{\mu}\rceil}]\, du\\
&+
\int\limits_{\lceil n^{\mu}\rceil}^{\infty} \frac{a_u}{c_u}c_u 
\sup_{\vartheta\in \mathbb{R}^{d}}
\E\left[\left\vert H(\vartheta,X_{\lfloor u\rfloor+1},c_k)-H(\vartheta,X_{\lfloor u\rfloor+1}^+,c_k)\right\vert\right]\, du
\\
&\leq C_2\frac{a_{\lceil n^\mu\rceil}}{c_{\lceil n^\mu\rceil}}\sum _{k=1}^{\infty}  k^{-1-\epsilon/2} 
\end{align*}
with some $C_{2}$ so
$$\E[\Sigma_3] 
=  O\left(n^{\mu(\gamma-1)} \right). $$ 

Combining the estimates we have so far, we get 
\begin{equation}\label{obudavar}
E[\Sigma_0+\Sigma_1'+\Sigma_2+\Sigma_3]= 
O(n^{-\mu}+n^{-2}+
n^{\frac{-\mu+2\mu\gamma-1}{2}}+
n^{-2}+n^{\mu(\gamma-1)}).
\end{equation}
Notice that $\E\vert \theta_t - \overline{y}_t\vert$  is always finite, see Lemma \ref{momentus} above.
Use Gronwall's lemma and \eqref{csillagocska} to obtain
the inequality
\begin{align*}
    \E[\vert \theta_{\lfloor t\rfloor} - \overline{y}_t\vert] &\leq 
    E[\Sigma_0+\Sigma_1'+\Sigma_{2}+\Sigma_{3}]
    \exp\left(C_{3}\int_{n^{\mu}}^{\lceil (n+1)^{\mu}\rceil} \frac{a_u}{c_u}\, du\right)
    \end{align*} 
    with some constant $C_{3}$.
From Lemma \ref{colis} it is also easy to check that 
$\E\vert \theta_t - \theta_{\lfloor t\rfloor}\vert=O(n^{-\mu})$.
Note furthermore that the terms $n^{-\mu}$ and $n^{\mu(\gamma-1)}$ are
always negligible in \eqref{obudavar}.
These observations lead to
\begin{align*}    
& \E\vert \theta_t - \overline{y}_t\vert \\
    & =O(n^\frac{-\mu+2\mu\gamma-1}{2}+ n^{-2})\exp\left(C_{4} n^{\mu\gamma-1} \right) \\
     & =  O(n^{\frac{1}{2}-\frac{1}{2\gamma}}+ n^{-2})
\end{align*}
with some $C_{4}$, finsihing the proof.
\end{proof}

\begin{proof}[Proof of Theorem \ref{thm:convergencrate}]
    Denote $$d_i=\sup_{\lceil i^{\mu}\rceil\leq s< \lceil (i+1)^{\mu}\rceil}
    \E\vert \theta_s- \overline{y}_s\vert,\ i=1,2,\ldots$${}
    By Fatou's lemma, we also have
    $$
    \E\vert \theta_{\lceil (i+1)^{\mu}\rceil}- \overline{y}_{\lceil (i+1)^{\mu}\rceil-}\vert\leq d_{i}
    $$
    where $\overline{y}_{s-}$ denotes the left limit of $\overline{y}$ at $s$.
    
    It follows from Lemma \ref{lem:dec_gain}, that $d_i=O(i^{-\beta})$. Combining this with  Assumption \ref{asp:stability} and using telescoping sums we get,
    for each integer $N\geq 1$,
    \begin{align*}
        \E\vert y(\lceil {N^{\mu}}\rceil ,1,\theta_{1})-\theta_{\lceil N^{\mu}\rceil}\vert &= \E\vert y(\lceil N^{\mu}\rceil ,1,\theta_1)-y(\lceil N^{\mu}\rceil,\lceil N^{\mu}\rceil ,\theta_{\lceil N^{\mu}\rceil})\vert \\
        &\leq \sum_{i=2}^{N} \E\left\vert y(\lceil N^{\mu}\rceil,\lceil(i-1)^{\mu}\rceil,
        \theta_{\lceil (i-1)^{\mu}\rceil})-y(\lceil N^{\mu}\rceil,\lceil i^{\mu}\rceil,\theta_{\lceil i^{\mu}\rceil})\right\vert \\
        &\leq \sum_{i=2}^{N} \E\left\vert y(\lceil N^{\mu}\rceil,\lceil i^{\mu}\rceil,y(\lceil i^{\mu}\rceil,\lceil (i-1)^{\mu}\rceil,\theta_{\lceil(i-1)^{\mu}\rceil}))-
        y(\lceil N^{\mu}\rceil,\lceil i^{\mu}\rceil,\theta_{\lceil i^{\mu}\rceil})\right\vert \\
        &\leq C^{*}\sum_{i=2}^{N} \left(\frac{i+1}{N} \right)^{\alpha\mu} d_{i-1} = O(N^{-\beta+1}),
    \end{align*}
    noting that $y(\lceil i^{\mu}\rceil,\lceil(i-1)^{\mu}\rceil,\theta_{
    \lceil (i-1)^{\mu}\rceil})$
    equals the left limit $\overline{y}_{\lceil i^{\mu}\rceil -}$.
    A similar argument provides, for all $t\in(\lceil N^{\mu}\rceil ,\lceil (N+1)^{\mu}\rceil)$, 
    $$\E\vert \theta_t- y(t,1,\theta_{1})\vert =O(N^{-\beta+1}).$$
    Taking $\mu$th root we obtain 
    $$\E\vert \theta_t- y(t,1,\theta_{1})\vert=O(t^{\frac{-\beta+1}{\mu}}),\ t\geq 1.$$
    To conclude, note that by the stability Assumption \ref{asp:stability},  
    $|y(t,1,\theta_{1})-\theta^{*}|\leq C^{*}|\theta_1-\theta^{*}|t^{-\alpha}$ and that $E|\theta_1|<\infty$,
    as easily seen using Lemma \ref{colis}.
\end{proof}

 \end{subsection}

\begin{subsection}{Fixed gain stochastic approximation}
Define $T=\frac{c}{a}$. For $nT\leq t<(n+1)T$, define $\overline{y}_t=y(t,nT,\theta_{nT})$, i.e. the solution of \eqref{ode_dec_gain} with the initial condition $y_{nT}=\theta_{nT}$. We use the piece-wise linear extension $\overline{\theta}_t$ of $\theta_t$ and the piece-wise constant extension $H(t,\theta)$ of $H(\theta,X_{k+1},c)$ as defined in the decreasing gain setting, but $a$ and $c$ are now constants.

\begin{lemma} \label{lem:fixed_gain}
Let Assumptions \ref{asp:G},\ref{asp:H-special_form}, \ref{smoothness}, 
\ref{asp:mix} and \ref{asp:stability} hold. Then for $t\in[nT,(n+1)T]$ there is $\overline{C}>0$ such that 
$\E|\theta_t-\overline{y}_{t}|\leq \overline{C} \max(c^2,\sqrt{\frac{a}{c}})$.
\end{lemma}

\begin{proof}


Using essentially the same estimates we derived in the decreasing gain setting, for fixed $a$ and $c$ we get 
\begin{align}
&\E[\Sigma_0]\leq C_0 a
\\
    &\E[\Sigma_1]\leq C_{1}\left[\frac{a}{c} \sum_{nT}^{t-1} \E|\theta_k-\overline{y_k}|  +c^{2}\right]\\
    &\E[\Sigma_2]\leq C_2 \left( \sum_{nT}^{t-1} \left(\frac{a^2}{c^2}\right) \right)^{1/2}\leq C_2 \left( \frac{c}{a}\frac{a^2}{c^2} \right)^{1/2} =C_2\sqrt{\frac{a}{c}}\\
    & \E[\Sigma_3]\leq C_{3}\left[\frac{a}{c}+\sum_{nT}^{t-1} ac\right] =C_3 Tac +C_{3}\frac{a}{c} =O\left(c^2+\frac{a}{c}\right),
\end{align}
with suitable constants $C_{0},C_{1},C_{2},C_{3}$.
Combine these estimates and use Gronwall-lemma to get
the statement.
To choose optimally, set $c^2=\sqrt{\frac{a}{c}}$, that is $c=a^{\frac{1}{5}}$. In this case $\E\vert\theta_t-\overline{y}_t\vert\leq C_4 a^{\frac{2}{5}}$ for some $C_4$. 
\end{proof}

\begin{proof}[Proof of Theorem \ref{thm:fixed_gain}]
Denote $$d_i=\sup_{iT\leq s<(i+1)T}\E\vert \theta_s- \overline{y}^i_s\vert.$$
    It follows from Lemma \ref{lem:fixed_gain}, that $d_i\leq \overline{C} \max(c^2,\sqrt{\frac{a}{c}})$. 
    Combining this with Assumption \ref{asp:stability} and using telescoping sums we get
    \begin{align*}
        \E\vert y(NT,1,\theta_1)-\theta_{NT}\vert &= \E\vert y(NT,1,\theta_1)-y(NT,NT,\theta_{NT})\vert \\
        &\leq \sum_{i=2}^{N} \E\left\vert y(NT,(i-1)T,\theta_{(i-1)T})-y(NT,iT,\theta_{iT})\right\vert \\
        &\leq \sum_{i=2}^{N} \E\left\vert y(NT,iT,y(iT,(i-1)T,\theta_{(i-1)T}))-y(NT,iT,\theta_{iT})\right\vert \\
       &\leq \sum_{i=2}^{N} \left(C^{*} e^{-a\alpha(NT-iT)} \right) d_{i-1}\leq  \hat{C} \max\left(c^2,\sqrt{\frac{a}{c}}\right),
    \end{align*}
    with some $\hat{C}$ since $\sum_{i=2}^N e^{-a\alpha(NT-iT)}$ has an upper bound independent of $N$. We similarly get

    $$\sup_{NT\leq t<(N+1)T}
    \E\vert \theta_t- {y}(t,1,\theta_{1})\vert\leq \check{C}\max\left(c^2,\sqrt{\frac{a}{c}}\right)$$
    with some $\check{C}$.
    To conclude, note that by the stability Assumption \ref{asp:stability},  
    $|y(t,1,\theta_{1})-\theta^{*}|\leq C^{*}|\theta_1-\theta^{*}|e^{-a\alpha t}$ and therefore 
    $$\E\vert \theta_t- \theta^{*}\vert=O\left(\max\left(c^2,\sqrt{\frac{a}{c}}\right)+
    e^{-a\alpha t}\right).$$ 
\end{proof}
\end{subsection}

\end{section}

\section{Auxiliary results}\label{auxiliary}

We define continuous-time analogues of the key quantities $M$ and $\Gamma$ from Assumption \ref{asp:mix}
and establish a pivotal
maximal inequality for them.

Consider a continuous-time filtration $(\mathcal{R}_t)_{t\in\mathbb{R}_+}$ as well as a decreasing family of sigma-fields
$(\mathcal{R}_t^+)_{t\in\mathbb{R}_+}$. We assume that $\mathcal{R}_t$ is
independent of $\mathcal{R}_t^+$, for all $t\in\mathbb{R}_+$.

We consider an $\mathbb{R}^{d}$-valued continuous-time stochastic process $(W_{t})_{t\in\mathbb{R}_+}$
which is progressively measurable (i.e.\ $W:[0,t]\times\Omega\to\mathbb{R}^d$ is $\mathcal{B}([0,t])\otimes\mathcal{R}_t$-measurable
for all $t\in\mathbb{R}_+$).
{}

From now on we assume that $W_{t}\in L^{1}$, $t\in\mathbb{R}_{+}$.
Fix $r\geq 1$. We define the quantities
\begin{eqnarray*}
	\tilde{M}_r &:=& \esssup_{t \in\mathbb{R}_+} 
	\E^{1/r}\left[|W_{t}|^{r}\vert {\mathcal{R}_{0}}\right],\\ 
	\tilde{\gamma}_r(\tau) &:=&  \esssup_{t\geq\tau}
	\E^{1/r}[|W_{t}- \E[W_{t}\vert {\mathcal{R}_{t-\tau}^+\vee \mathcal{R}_0}]|^r\vert {\mathcal{R}_{0}}],\ \tau\in\mathbb{R}_+ ,
\end{eqnarray*}
and set $\tilde{\Gamma}_r := \sum_{\tau=0}^{\infty}\tilde{\gamma}_r(\tau)$.

Now we recall a powerful maximal inequality,
Theorem B.3 of \cite{bernoulli}.
 
\begin{theorem}\label{elso} Let $(W_{t})_{t \in \mathbb{R}_{+}}$ be $L^{r}$-bounded for some $r> 2$
and let $\tilde{M}_{r}+\tilde{\Gamma}_{r}<\infty$ a.s.
Assume $\E[{W_{t}}\vert{\mathcal{R}_{0}}]=0$ a.s.\ for $t\in\mathbb{R}_+$.
Let $f:[0,T]\to\mathbb{R}$ be $\mathcal{B}([0,T])$-measurable with $\int_{0}^{T}f_{t}^{2}\, d t<\infty$. 
Then there is a constant $C'(r)$ such that
\begin{equation}\label{erd}
\E^{1/r}\left[\sup_{s\in [0,T]}\left|\int_{0}^{s} f_t W_t\, dt \right|^r\vert {\mathcal{R}_{0}}\right]
\leq C'(r)\left( \int_{0}^{T} f_t^{2}\, dt \right)^{1/2} [\tilde{{M}}_r +\tilde{\Gamma}_r],
\end{equation}
almost surely.\hfill $\Box$
\end{theorem}

We also recall Lemma A.1 of \cite{chau2016fixed}.

\begin{lemma}\label{mall} Let 
$\mathcal{G},\mathcal{H}\subset\mathcal{F}$
be sigma-algebras. Let $X,Y\in\mathbb{R}^{d}$ be random variables in $L^p$ such that $Y$ is measurable with
	respect to $\mathcal{H}\vee\mathcal{G}$.
	Then for any $p\ge 1$,
	$$
	E^{1/p}\left[\vert X-E[X\vert\mathcal{H}\vee\mathcal{G}]\vert^p\big\vert \mathcal{G}\right]
	\leq 2E^{1/p}\left[\vert X-Y\vert^p\big\vert \mathcal{G}\right].
	$$
	\hfill $\Box$
\end{lemma}

\begin{lemma}\label{wales} Let the process $W_{t}$ be defined by \eqref{defw}.  
Taking the filtration $\mathcal{R}_{t}:=\mathcal{F}_{t+\lceil n^{\mu}\rceil}$ and
$\mathcal{R}^+_t:=\mathcal{F}_{t+\lceil n^{\mu}\rceil}^+$,
we get $\tilde{M}(W)+\tilde{\Gamma}(W)\leq C(1+D^{2})$
for some $C>0$.
\end{lemma}
\begin{proof}
Estimations of  
Lemma \ref{lem:M-Gamma-estimate} with $R=D$ and with $\overline{y}_{t+\lceil n^{\mu}\rceil}$ instead
of $\theta$ imply the statement.	
\end{proof}

\section{Numerical experiments}\label{opt}
In what follows we present numerical results to check 
the convergence of the algorithm for a simple discontinuous function $J$, defined as

$$J(\theta,X)=
\begin{cases}
(\theta-X)^2 +1 \text{, if }X\leq\theta\\
(\theta-X)^2\text{, otherwise,}
\end{cases}$$
where $X$ is a square-integrable, absolutely continuous random variable. Clearly,
this function is not continuous in the parameter, but its
expectation \emph{is} continuous:
\begin{align*}
       U(\theta) &= \E J(X,\theta) =\int_{-\infty}^{\theta} ((x-\theta^2)+1)f(x)dx + \int_{\theta}^{\infty}(x-\theta)^2f(x)dx    \\
    & = \E (X-\theta)^2 + F(\theta)=\E X^2 -2\theta \E X +\theta^2 + F(\theta), 
\end{align*}
where $f(\cdot)$ and $F(\cdot)$ are the density function resp.\ the distribution function of $X$. 
Assuming that $F$ is differentiable, we need to solve 
$$\frac{\partial  U(\theta) }{\partial\theta}=-2\E X + 2\theta -f(\theta)=0$${}
in order to find $\argmin\E J(X,\theta)$.  

For the numerical examples we will use the recursion
\begin{align} \label{recursion_for_numerics}
    \theta_{k+1}=\theta_k + \frac{1}{k+k_0}\frac{J(\theta+(k+k_0)^{-1/5},X_{k+1})
-J(\theta-(k+k_0)^{-1/5},X'_{k+1})}{(k+k_0)^{-1/5}}.
\end{align}

To compute the expected error, Monte Carlo simulations were used with 
$10000$ sample paths and the number of steps $k$ ranging from $2^{8}$ to $2^{20}$. 
We fit regression on the log-log plot to get the convergence rate only on 
$[2^{13},2^{20}]$ and set $k_0=10000$ to avoid the initial fluctuations of the algorithm.

\subsection{Independent innovations}

In this section we assume that the consecutive 
``measurement noises'' $X_{n}$ are i.i.d. We consider three different choices for the distribution 
of the noise: normal, uniform and beta distributions. Note that normal distribution 
violates boundedness and for uniform distribution the differentiability of $F$ fails, however 
convergence is achieved even in these cases. We also distinguish between the case where the observations 
$X_{k+1}$ and $X'_{k+1}$ are the same and when they are independent. Here we refer back to 
Remark \ref{rem:X_same_or_different} where we point out that this choice does not influence 
our theoretical results, however it may make a visible difference numerically. This phenomenon 
has already been observed, see \cite{glasserman1992some} for more about the variance 
reduction technique called \textit{common random numbers} (CRN). The values in 
Table \ref{table:convergence_speed_iid} below represent the slope of linear regression we fit on the log-log plot of the average absolute error vs. the number of steps, together with the R-squared value measuring the goodness of the fit.

\begin{table}[h]
\begin{centering}

\begin{tabular}{l|l|l|}
\cline{2-3}
                                      & independent $X_{k+1}$,$X'_{k+1}$  & identical $X_{k+1}=X'_{k+1}$   \\ \hline
\multicolumn{1}{|l|}{N$(0,1)$}        & $-0.299$ $(R^2=0.999)$     & $-0.459$ $(R^2=0.999)$ \\ \hline
\multicolumn{1}{|l|}{U$([0,1])$}      &  $-0.14$ $(R^2=0.997) $          &  $-0.14$ $(R^2=0.997)$     \\ \hline
\multicolumn{1}{|l|}{Beta$(2,2)$} &  $-0.374 $ $(R^2=0.999)$          &     $-0.393 $ $(R^2=0.999)$       \\ \hline
\end{tabular}
\caption{Convergence speed for different distributions of i.i.d. noise}
\label{table:convergence_speed_iid}
\end{centering}
\end{table}

The lower limit that we theoretically achieved for the convergence rate in Theorem \ref{thm:convergencrate} was ${-0.2}$, however the numerical experiments we present show that the practical convergence rate can outperform this.


\subsubsection{Standard normal distribution}

Assume that $X\sim N(0,1)$. Then the function we aim to find the minimum of is 
$U_1(\theta) = 1+\theta^2+\Phi(\theta),$ where $\Phi$ denotes the the cumulative distribution function of standard normal distribution. We get the solution $\theta^{*} =-\sqrt{W\left(\frac{1}{8\pi}\right)}\approx -0.19569,$ where $W$ is the Lambert-W function. 


Figure \ref{fig:convergence_ind} illustrates the convergence of two variations of algorithm (\ref{recursion_for_numerics}) for $U_1$, starting the iteration from $\theta_0=-0.1$. On figure (\ref{fig:convergence_ind_ind}) we present the case where $X_{k+1}$ and $X'_{k+1}$ are independent on a log-log plot, we observe a convergence rate of $k^{-0.299}$ while (\ref{fig:convergence_ind_same}) shows the case where $X_{k+1}=X'_{k+1}$ which yields a convergence rate of $k^{-0.459}$. 


\begin{figure}[ht]
\begin{subfigure}{0.5\textwidth}
\includegraphics[width=1\linewidth, height=5.5cm]{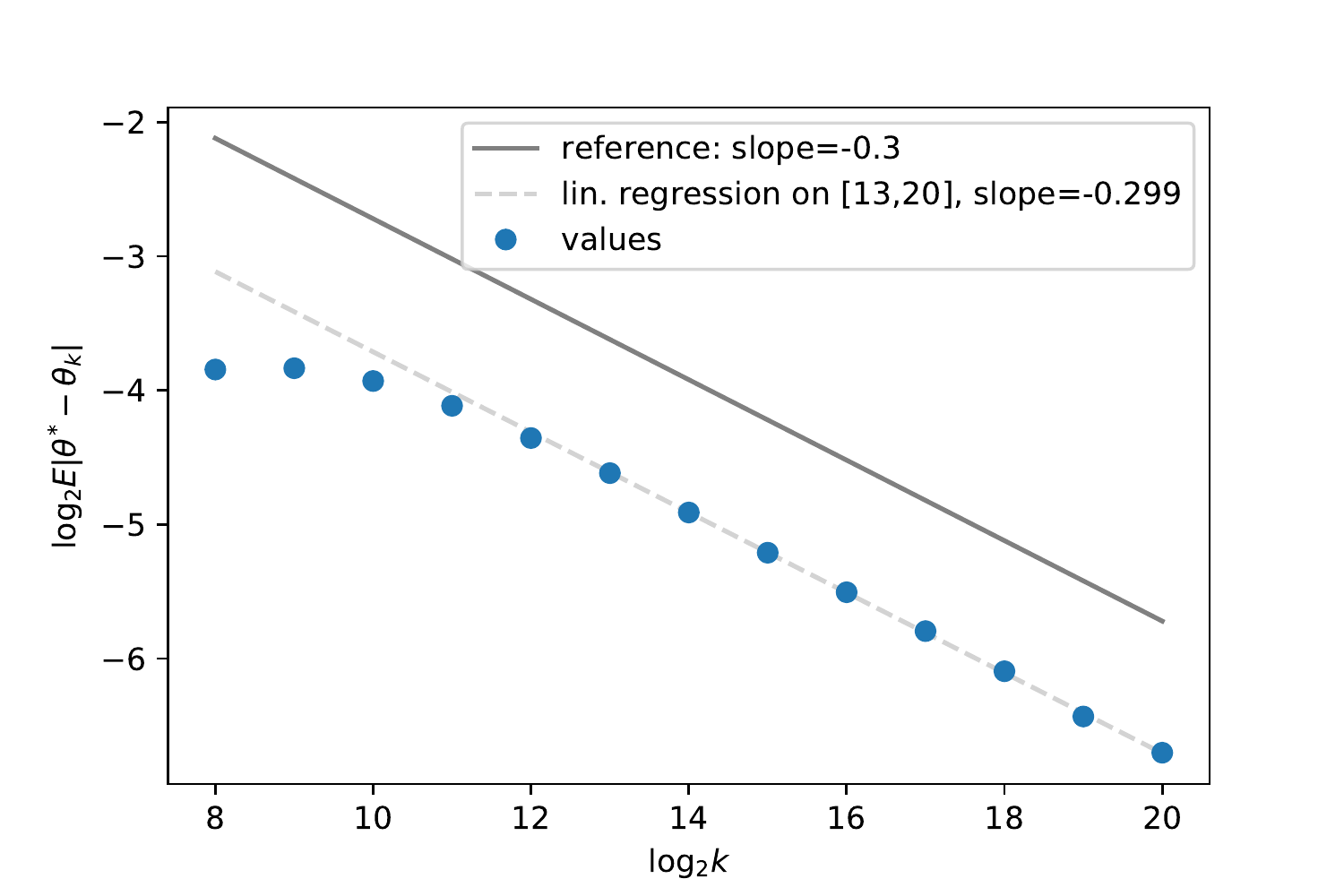}
\caption{$X_{k+1}$ and $X'_{k+1}$ independent}
\label{fig:convergence_ind_ind}
\end{subfigure}
\begin{subfigure}{0.5\textwidth}
\includegraphics[width=1\linewidth, height=5.5cm]{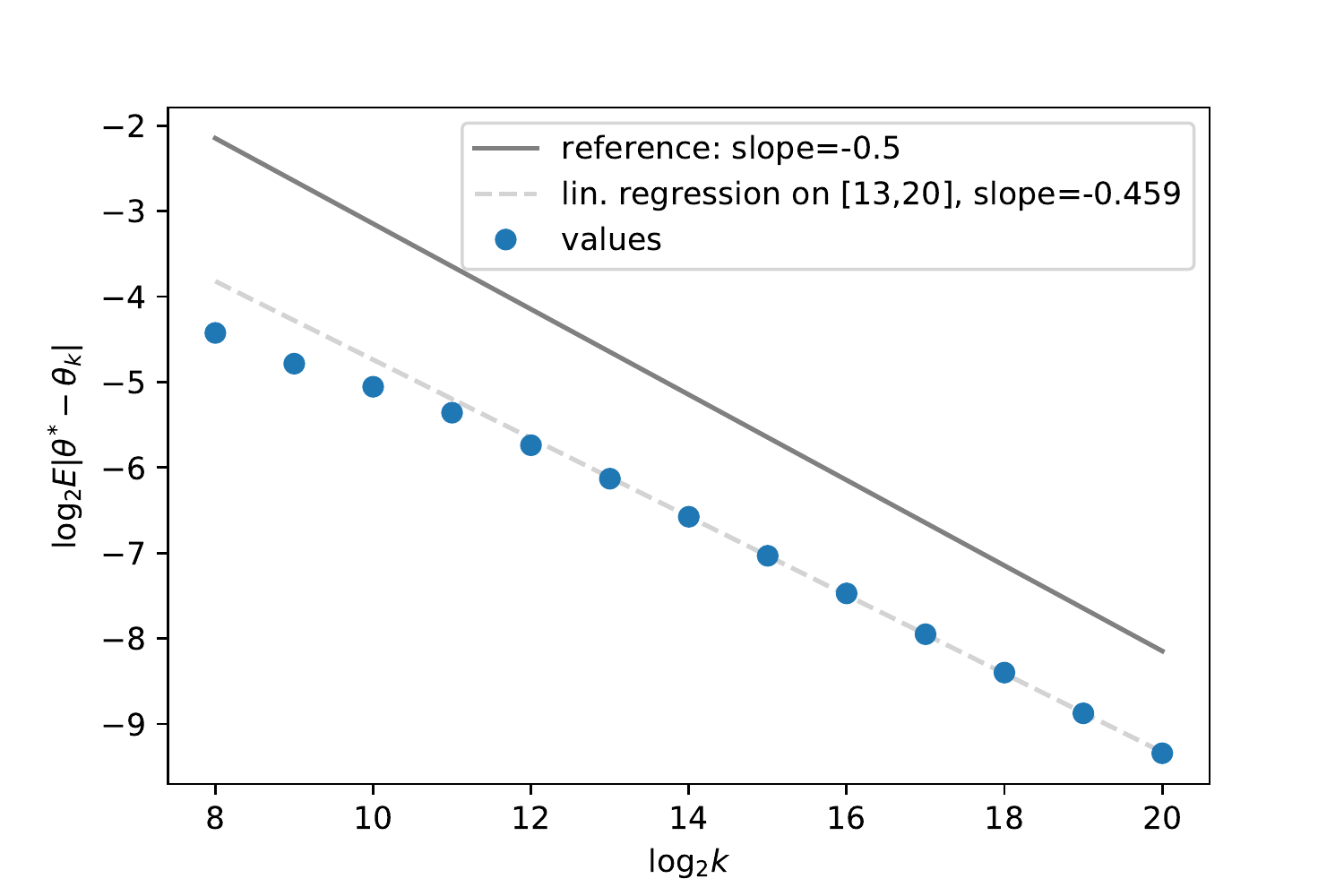}
\caption{ $X_{k+1}=X'_{k+1}$}
\label{fig:convergence_ind_same}
\end{subfigure}
\caption{Log-log plot of $\E|\theta^*-\theta_k|$ vs. number of iterations for i.i.d. standard normal innovations }
\label{fig:convergence_ind}
\end{figure}

 \subsubsection{Uniform([0,1]) distribution}

 Let $X\sim Uniform([0,1])$. Then the function we aim to find the minimum of is
 $U_2(\theta) = 1/3-\theta+\theta^2+F_{uni}(\theta),$ where $F_{uni}$ denotes the the cumulative distribution function of Uniform([0,1]) distribution. 
 We get the solution $\theta^{*} =0.$ 
 
 Figure \ref{fig:convergence_uni} illustrates the convergence of two variations of algorithm (\ref{recursion_for_numerics}) for $U_2$, starting the iteration from $\theta_0=1$. On figure (\ref{fig:uni_ind}) we present the case where $X_{k+1}$ and $X'_{k+1}$ are independent on a log-log plot, while (\ref{fig:uni_same}) shows the case where $X_{k+1}=X'_{k+1}$, both of which yield a convergence rate of $k^{-0.14}$, worse
 that the theoretical rate $k^{-0.2}$. 

\begin{figure}[ht]
 \begin{subfigure}{0.5\textwidth}
 \includegraphics[width=1\linewidth, height=5.5cm]{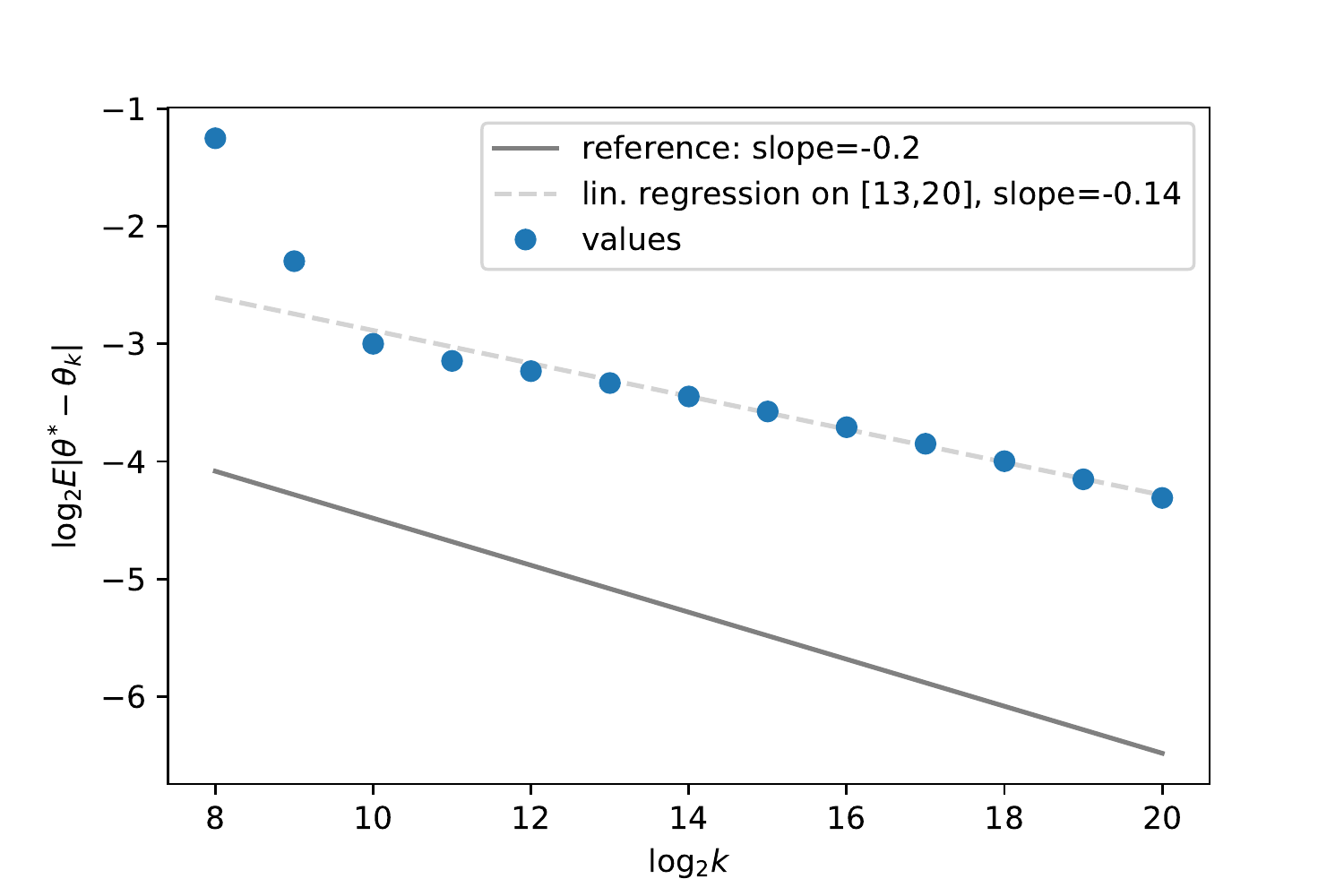}
 \caption{$X_{k+1}$ and $X'_{k+1}$ independent}
 \label{fig:uni_ind}
 \end{subfigure}
 \begin{subfigure}{0.5\textwidth}
 \includegraphics[width=1\linewidth, height=5.5cm]{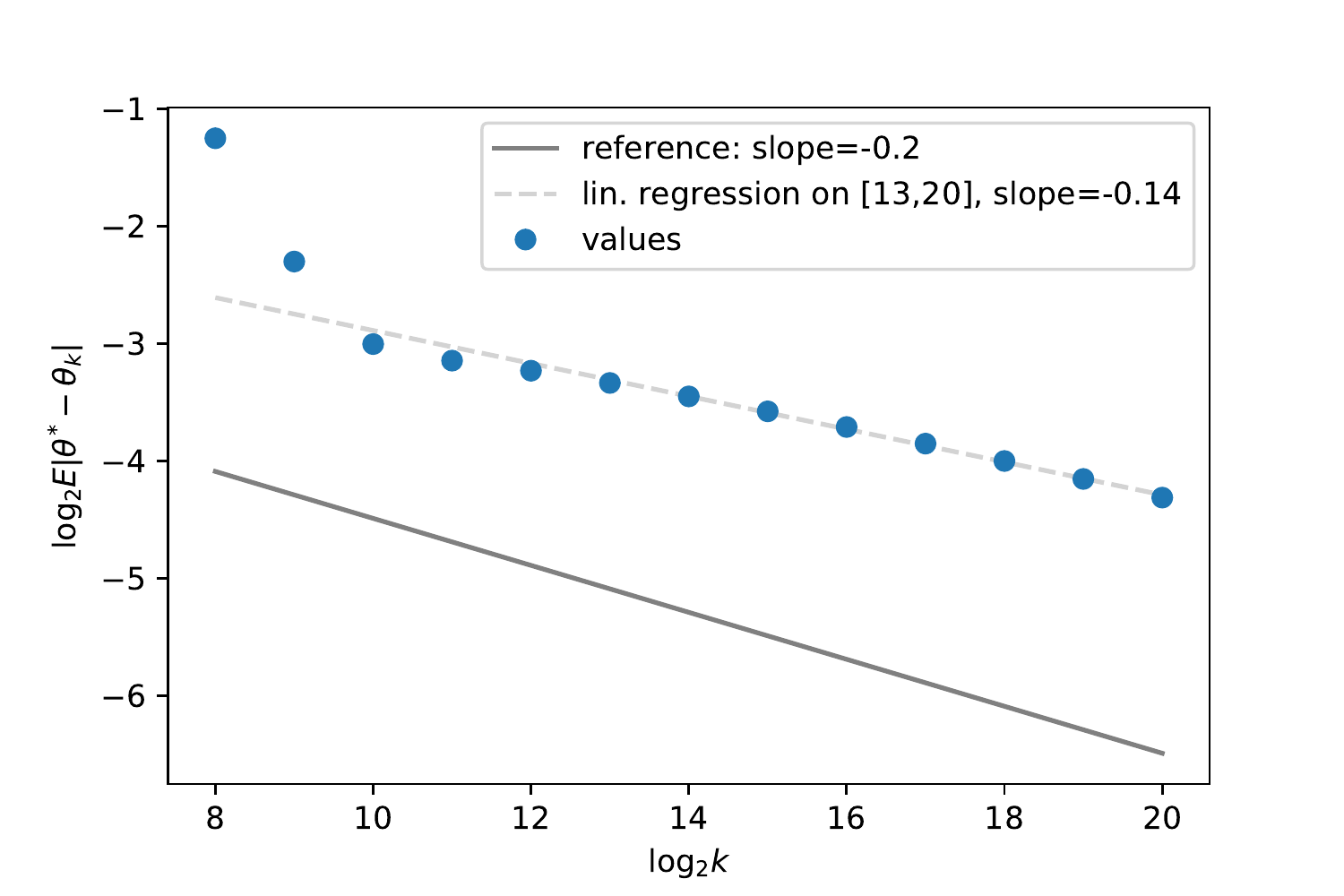}
 \caption{ $X_{k+1}=X'_{k+1}$}
 \label{fig:uni_same}
 \end{subfigure}
 \caption{Log-log plot of $\E|\theta^*-\theta_k|$ vs. number of iterations for i.i.d. uniform innovations }
 \label{fig:convergence_uni}
 \end{figure}

 \subsubsection{Beta(2,2) distribution}

 Let $X\sim Beta(2,2)$. Then the function we aim to find the minimum of is
 $U_3(\theta) = 0.3-\theta+\theta^2+F_{\beta}(\theta),$ where $F_{\beta}$ denotes the the cumulative distribution function of Beta(2,2) distribution. 
 We get the solution $\theta^{*} =\frac{2-\sqrt{2.5}}{3}\approx 0.13962.$

Figure \ref{fig:convergence_beta} illustrates the convergence of two variations of algorithm (\ref{recursion_for_numerics}) for $U_3$, starting the iteration from $\theta_0=1$. On figure (\ref{fig:convergence_beta_ind}) we present the case where $X_{k+1}$ and $X'_{k+1}$ are independent on a log-log plot, we observe a convergence rate of $k^{-0.374}$ while (\ref{fig:convergence_beta_same}) shows the case where $X_{k+1}=X'_{k+1}$ which yields a convergence rate of $k^{-0.393}$. 

 \begin{figure}[ht]
 \begin{subfigure}{0.5\textwidth}
 \includegraphics[width=1\linewidth, height=5.5cm]{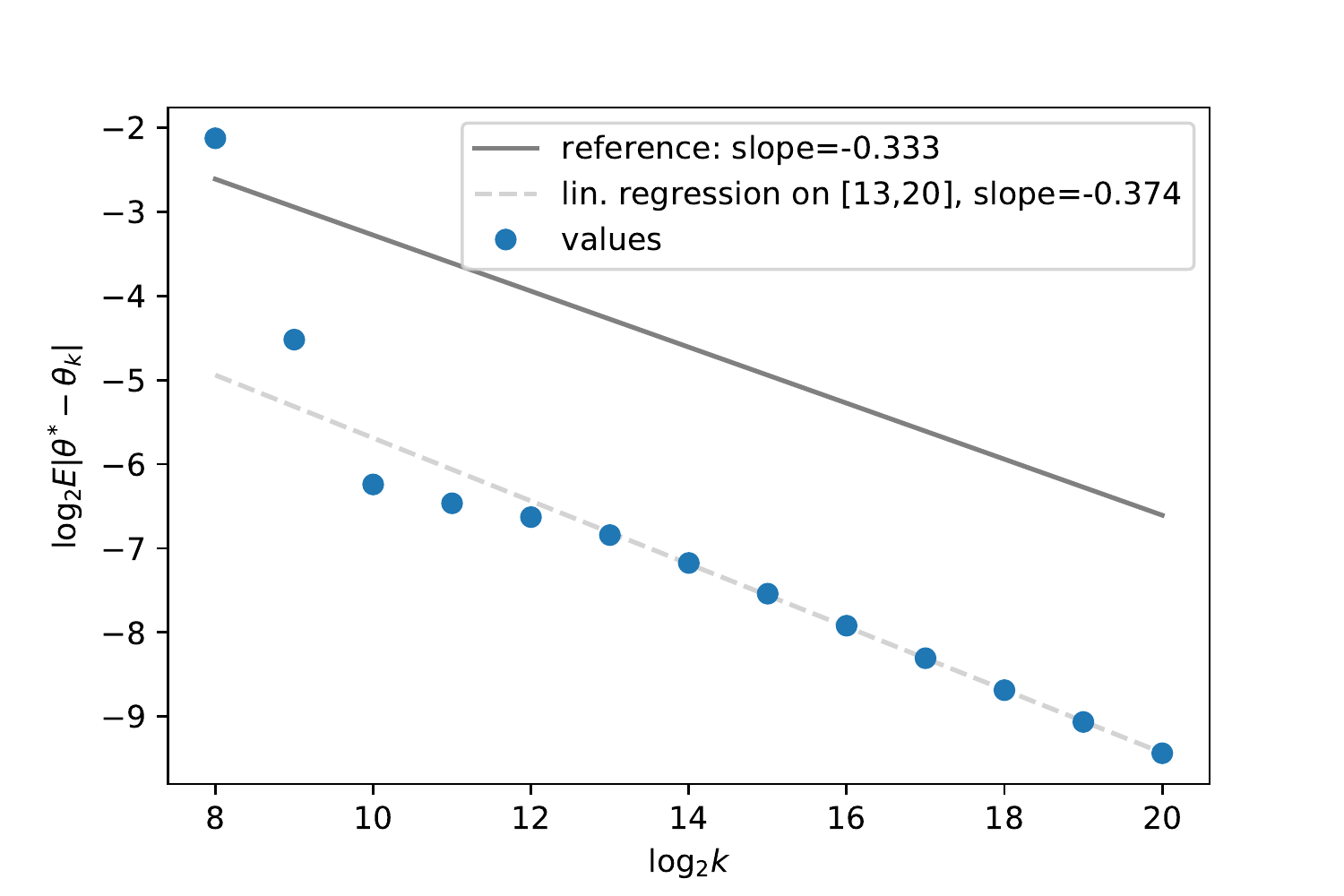}
 \caption{$X_{k+1}$ and $X'_{k+1}$ independent}
 \label{fig:convergence_beta_ind}
 \end{subfigure}
 \begin{subfigure}{0.5\textwidth}
 \includegraphics[width=1\linewidth, height=5.5cm]{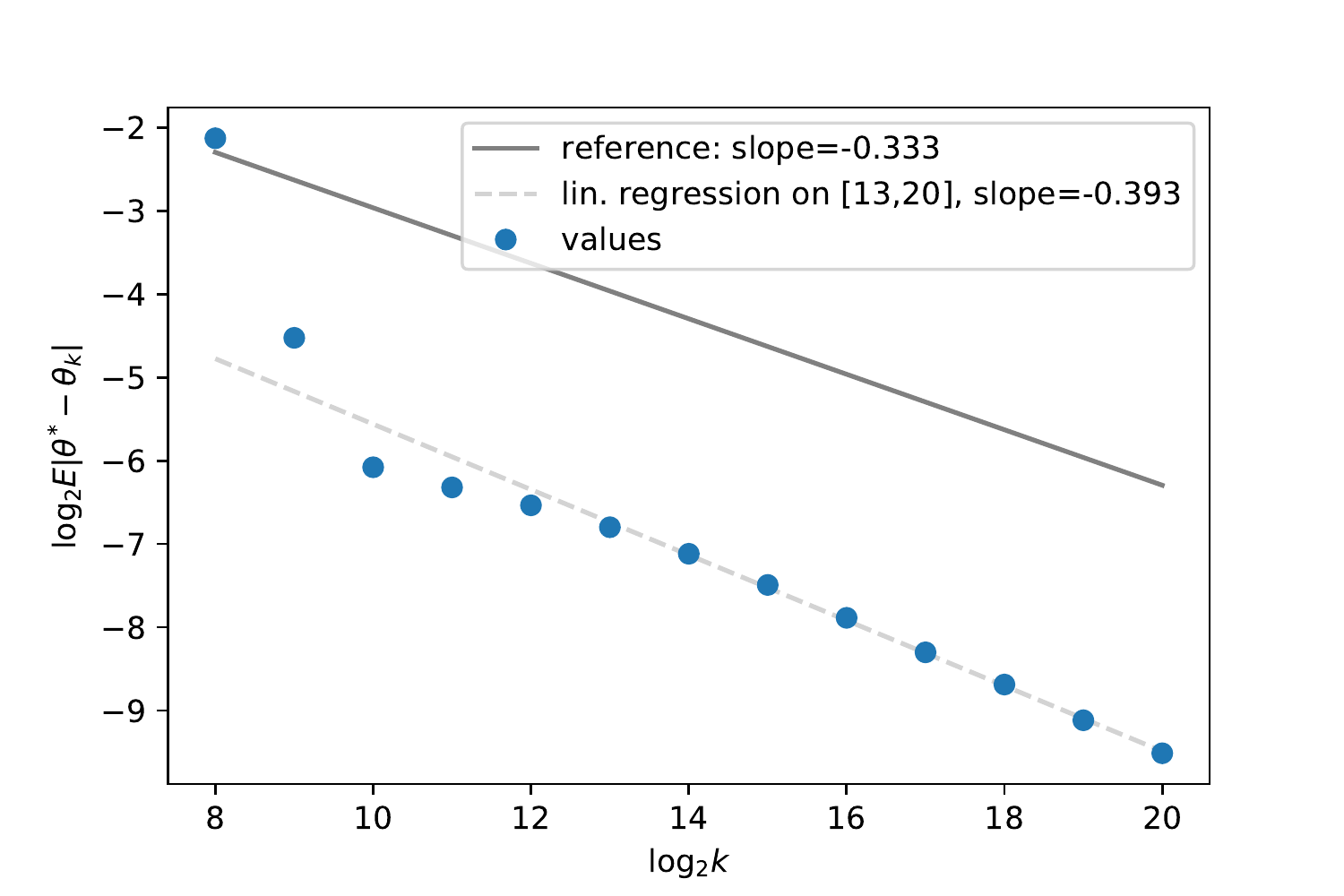}
 \caption{ $X_{k+1}=X'_{k+1}$}
 \label{fig:convergence_beta_same}
 \end{subfigure}
 \caption{Log-log plot of $\E|\theta^*-\theta_k|$ vs. number of iterations for i.i.d. beta innovations }
 \label{fig:convergence_beta}
 \end{figure}

\subsection{AR(1) innovations}


For an example with non-i.i.d.\ $X_{t}$, assume that the ``noise'' 
is an AR(1) process defined as
$$Y_{t+1}= \kappa Y_t + \varepsilon_{t+1}\text{, for }t\in\mathbb{Z}, $$
where $\varepsilon_t$ is standard normal for $t\in\mathbb{Z}$ and $\vert\kappa\vert<1$. Clearly, $Y_{t}= \sum_{k=0}^\infty \kappa ^k\varepsilon_{t-k},$ 
and therefore $Y_t\sim N\left(0, \frac{1}{1-\kappa^2}\right)$. For the sequences $X_t$ and $X_t'$ 
we have two options: either we take consecutive measurements i.e.\ $X_{k}=Y_{2k-1}$ and 
$X'_{k}=Y_{2k}$ or we use identical values, i.e. $X_{k}=X'_{k}=Y_k$. In both cases
\begin{align*}
    U_4(\theta)=\E J(\theta, X) =\theta^2+\frac{1}{1-\kappa^2}+\Phi\left(\theta\sqrt{1-\kappa^2}\right).
\end{align*}
Solving this for $\kappa=0.75$ we get the optimal value
$\theta^{*}\approx -0.13144.$

Figure \ref{fig:convergence_ar_7} and table \ref{table:AR_conv} illustrate the convergence rate of algorithm (\ref{recursion_for_numerics}) for the function $U_4$, starting from $\theta_0=0$. On figure (\ref{fig:convergence_ar_ind}) we present the rate in the case where we take consecutive measurements of the AR(1) process, ($X_{k}=Y_{2k-1}$ and $X'_{k}=Y_{2k}$), the convergence rate of $k^{-0.333}$ was observed. Figure (\ref{fig:convergence_ar7_same}) shows the case where the two measurements are the same,  ($X_{k}=X'_{k}=Y_k$), with the rate $k^{-0.487}$.

\begin{table}[h]
\begin{centering}
\begin{tabular}{l|l|l|}
\cline{2-3}
                            & consecutive observations:  $X_{k}=Y_{2k-1}$ and $X'_{k}=Y_{2k}$ & identical: $X_{k}=X'_{k}=Y_k$ \\ \hline
\multicolumn{1}{|l|}{AR(1)} & -0.333 $(R^2=0.999)$                        & -0.487 $(R^2=0.999)$                   \\ \hline
\end{tabular}
\caption{Convergence rate for AR(1) noise}
\label{table:AR_conv}
\end{centering}
\end{table}

\begin{figure}[h]
\begin{subfigure}{0.5\textwidth}
\includegraphics[width=1\linewidth, height=5.5cm]{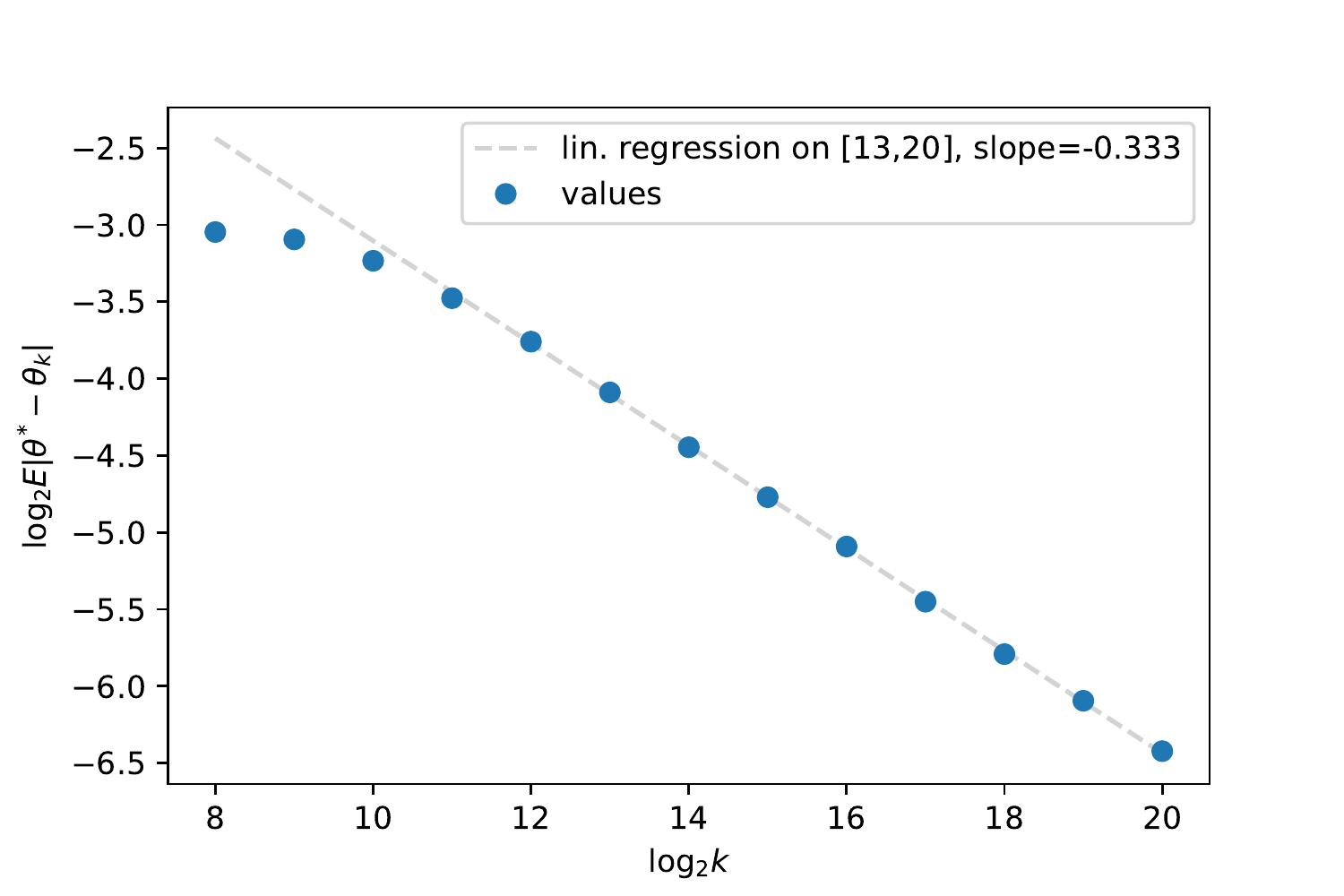}
\caption{$X_{k}=Y_{2k-1}$ and $X'_{k}=Y_{2k}$}
\label{fig:convergence_ar_ind}
\end{subfigure}
\begin{subfigure}{0.5\textwidth}
\includegraphics[width=1\linewidth, height=5.5cm]{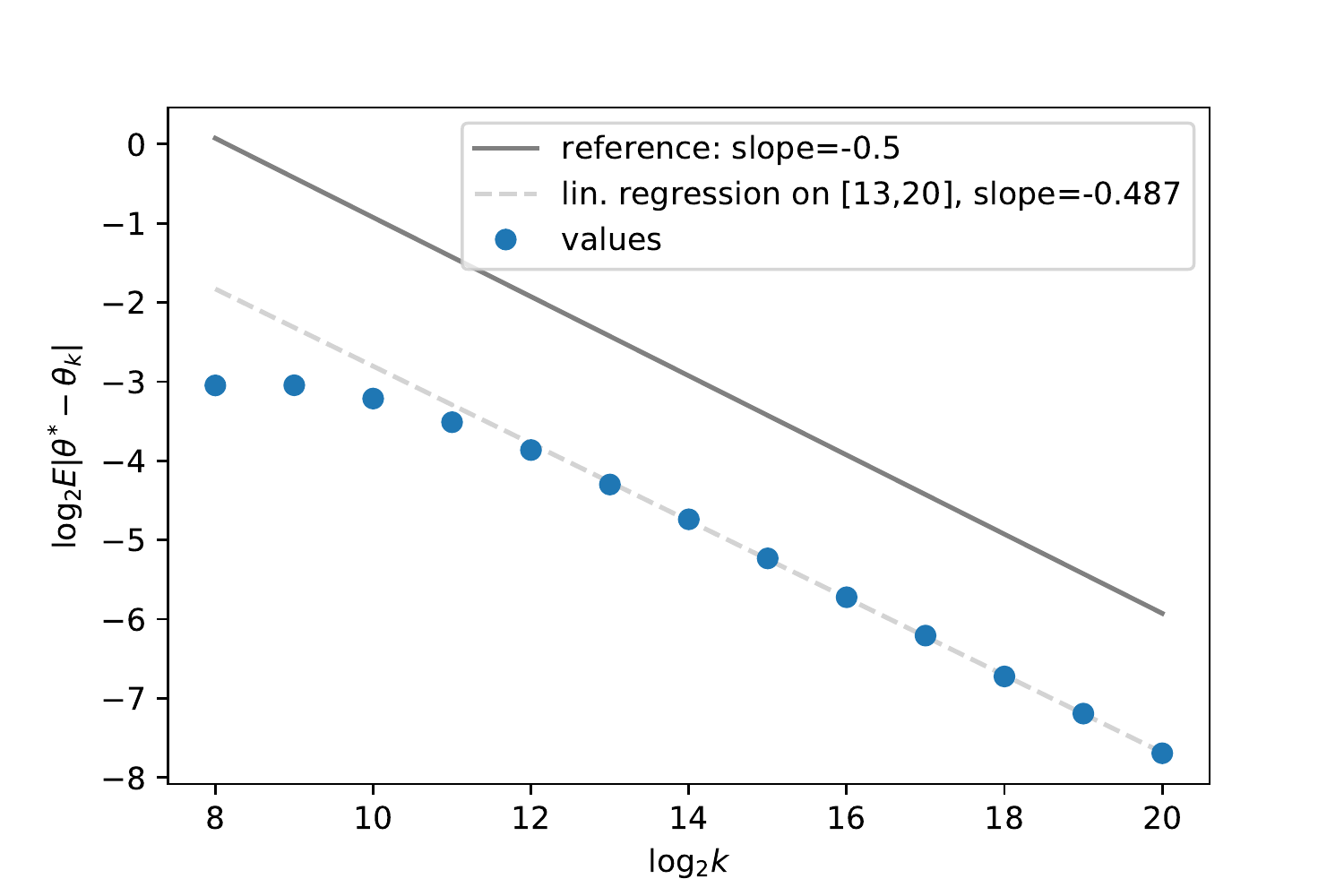}
\caption{ $X_{k}=X'_{k}=Y_k$}
\label{fig:convergence_ar7_same}
\end{subfigure}
\caption{Log-log plot of $\E|\theta^*-\theta_k|$ vs. number of iterations for AR(1) innovations }
\label{fig:convergence_ar_7}
\end{figure}



\section{Application to mathematical finance}\label{sec:mathfin}

The price of a financial asset either follows a trend during a given period of
time or just rambles around its ``fair'' price value -- at least so it seems to
many actual traders. This ``rambling'', in more mathematical
terms, means that the price is reverting to its long-term average. Such a mean-reversion phenomenon
can be exploited by ``buying low, selling high''-type strategies. Related discussions 
involve plenty of common-sense
advice and benevolent concrete suggestion, see e.g.\ \cite{www1,www2,www3}. 
There exist also theoretical studies about optimal trading with such prices,
see e.g.\ \cite{guasoni}. However, a rigorous approach to \emph{adaptive} trading
algorithms of this type is lacking.

Results of the present paper provide theoretical convergence guarantees for such algorithms
which cannot be deduced from existing literature on stochastic approximation. The most
conspicuous feature of mean-reversion strategies is that they are triggered when
the price reaches a certain level. This means that their payoffs are \emph{discontinuous}
with respect to the parameters, gradients do not exist and only finite-difference approximations
can be used (the Kiefer-Wolfowitz method). Their convergence in the given discontinuous case
cannot be shown based on available results hence we fill an important and practically relevant
gap here.

We describe in some detail a trading model below and explain how it fits into the
framework used in the previous sections. Let the price of the observed financial
asset be described by a real-valued stochastic process $S_{t}$, $t\in\mathbb{Z}$, adapted
to a given filtration $\mathcal{F}_{t}$, $t\in\mathbb{Z}$, representing the flow of information.
(Alternatively $S_{t}$ may be the \emph{increment} of the price at $t$ which can safely
be assumed to follow a stationary process.) 

Our algorithm will be based on several dynamically updated estimators which are
assumed to be functionals of the trajectories of $S_{t}$ and possibly of another adapted process $F_{t}$
describing important economic factors. The estimate for the long-term average of the process
is denoted by $A_{t}(\theta)$ at time $t$. The upper and lower bandwith processes will be denoted 
$B_{t}^{+}(\theta)$
and $B_{t}^{-}(\theta)$, they are non-negative. All these estimates depend on a parameter
$\theta$ to be tuned, where $\theta$ ranges over a subset $Q$ of $\mathbb{R}^{d}$.

In practice, $A_{t}(\theta)$ is some moving average (or exponential moving average) of previous
values of $S$ but it may depend on the other indicators $F$ (market indices, etc.). Here $\theta$ determines,
for instance, the weights of the moving average estimate. The quantities
$B_{t}^{\pm}(\theta)$ are normally based on standard deviation estimates for $S$ but, again, may be more
complex with $\theta$ describing weighting of past information. If we peek from time $t$ back to
time $t-p$ with some $p\in\mathbb{N}$ then $A_{t}(\theta),B^{\pm}_{t}(\theta)$ are
functionals of $(S_{t-p},F_{t-p},\ldots,S_{t},F_{t})$.

The price range $[A_{t}-B^{-}_{t},A_{t}+B^{+}_{t}]$ is considered to be ``normal'' by the algorithm while
quitting that interval suggest ``extremal'' behaviour that the market should correct soon.
For example, reaching the level $A_{t}-B^{-}_{t}$ means that the price is abnormally low for
the present circumstances, hence it is worth buying a quantity $b(\theta)$ of them where, again,
the parameter $\theta$ should be optimally found. When the price returns to $A_{t'}$ at some
later time $t'$, the asset will be sold and a profit is realized. Similarly,
when reaching $A_{t}+B^{+}_{t}$, quantity $s(\theta)$ of the asset is sold (the price
being abnormally high) and it will be repurchased once the ``normal'' level $A_{t'}$ is reached
at some future $t'>t$, aiming to realize profit.

The value of the parameter $\theta$ will be updated at times $tN$, $t\in\mathbb{N}$ where
$N\geq 1$ is fixed. The (random) profit (or loss) resulting from trading on the
interval $[N(t-1),Nt]$ is denoted by $u(\theta,X_{t})$ with 
$X_{t}=(S_{N(t-1)-p},F_{N(t-1)-p},\ldots,S_{Nt},F_{Nt})$. We could even write an explicit
expression for $u$ based on the description of the trading mechanism in the previous
paragraph but it would be very cumbersome without providing additional insight hence we omit it.
We also add that, in many cases, a fee must also be paid at every transaction.
Such strategies being ``threshold-type'', the function $u$ is generically
a \emph{discontinuous} function of $\theta$. 

We furthermore argue that one \emph{cannot} smooth out $u$ and make it continuous without
losing \emph{essential} features of the problem. Approximating the indicator function of the
interval $[0,\infty)$ by a function $f$ which is $1$ on $[0,\infty)$, $0$ on $(-\infty,-\epsilon]$
for some small $\epsilon>0$ and linear on $(-\epsilon,0)$ may look reasonable at first sight but
in this way we get a Lipschitz approximation with a huge Lipschitz constant hence with a poor
convergence rate! This is just to stress that such simple tricks might work in certain practical situations but
they only obscure the real issues in the theoretical analysis (namely, there \emph{is} a discontinuity
to be handled).

The described algorithm is very close to what actual investors
do, see \cite{www1,www2,www3}. We also mention the related theoretical studies \cite{leung,cartea} which,
however, do not take an adaptive view and calculate optimal strategies for concrete models.

Taking a more realistic, adaptive approach, the investor
may seek to maximize $Eu(\theta,X_{0})$ by dynamically updating $\theta$ 
at every instant $tN$, $t\in\mathbb{N}$. Our versions of the Kiefer-Wolfowitz
algorithm, presented in the previous sections, are tailor-made for such online
optimization, both the decreasing and the fixed gain version, depending on
the circumstances. Theorems \ref{thm:convergencrate} and 
\ref{thm:fixed_gain} 
provide solid theoretical convergence
guarantees for such procedures.


\begin{thebibliography}{10}

\bibitem{bernoulli}
M. Barkhagen, N. H. Chau, {\'E}. Moulines, M. R{\'a}sonyi,
  S. Sabanis, and Y. Zhang.
\newblock On stochastic gradient {L}angevin dynamics with dependent data
  streams in the logconcave case.
\newblock {\em Bernoulli}, 27:1--33, 2021.


\bibitem{bmp}
A.~Benveniste, M.~M\'etivier, and P.~Priouret.
\newblock {\em Adaptive algorithms and stochastic approximations}.
\newblock Springer, 1990.

\bibitem{cartea}
\'A. Cartea, S. Jaimungal and J. Penalva.
\newblock Algorithmic and high-frequency trading.
\newblock Cambridge University Press, 2015.

\bibitem{chau2016fixed}
H.~N Chau, C. Kumar, M. R{\'a}sonyi, and S. Sabanis.
\newblock On fixed gain recursive estimators with discontinuity in the
  parameters.
\newblock {\em ESAIM: Probability and Statistics}, 23:217--244, 2019.

\bibitem{chau2019stochastic}
N.~H. Chau, {\'E}. Moulines, M. R{\'a}sonyi, S. Sabanis, and Y. Zhang.
\newblock On stochastic gradient {L}angevin dynamics with dependent data
  streams: the fully non-convex case.
\newblock {\em To appear in SIAM Journal on Mathematics of Data Science}, 2021.
arXiv:1905.13142

\bibitem{eric}
A. Durmus and \'E. Moulines.
\newblock Nonasymptotic convergence analysis for the unadjusted {L}angevin
  algorithm.
\newblock {\em Annals of Applied Probability}, 27:1551--1587, 2017.


\bibitem{fort}
G.~Fort, \'E. Moulines, A.~Schreck, and M.~Vihola.
\newblock Convergence of Markovian stochastic 
approximation with discontinuous
  dynamics.
\newblock {\em SIAM Journal on Control and Optimization}, 
54(2):866–893, 2016.

\bibitem{gerencser1989class}
L. Gerencs{\'e}r.
\newblock On a class of mixing processes.
\newblock {\em Stochastics}, 26(3):165--191, 1989.




\bibitem{gerencser1992rate}
L. Gerencs{\'e}r.
\newblock Rate of convergence of recursive estimators.
\newblock {\em SIAM Journal on Control and Optimization}, 30(5):1200--1227,
  1992.

\bibitem{laci_spsa1}
L. Gerencs{\'e}r.
\newblock Convergence rate of moments in stochastic 
approximation with simultaneous perturbation gradient approximation and resetting.
\newblock \emph{IEEE Trans. Automatic Control}, 44:894--905, 1999.

\bibitem{laci_spsa2}
L. Gerencs{\'e}r.
\newblock SPSA with state-dependent noise. A tool for direct adaptive control.
\newblock \emph{In: Proc. of the 37th IEEE Conference on Decision
and Control, Tampa, Fl., USA, (IEEE, 1998)}, 3451--3456, 1998.

\bibitem{glasserman1992some}
P. Glasserman and D. D. Yao.
\newblock Some guidelines and guarantees for common random numbers.
\newblock {\em Management Science}, 38(6):884--908, 1992.


\bibitem{guasoni} P. Guasoni, A. Tolomeo and G. Wang.	
\newblock Should Commodity Investors Follow Commodities' Prices?
\newblock \emph{SIAM Journal on Financial Mathematics}, 10:466--490, 2019.


\bibitem{kiefer1952stochastic}
J. Kiefer and J. Wolfowitz.
\newblock Stochastic estimation of the maximum of a regression function.
\newblock {\em Annals of Mathematical Statistics}, 23(3):462--466, 1952.

\bibitem{kushner}
H. J. Kushner and D. S. Clark.
\newblock \emph{Stochastic approximation for constrained and unconstrained systems.}
\newblock Springer, 1978.

\bibitem{laruelle2012stochastic}
S. Laruelle and G. Pag{\`e}s.
\newblock Stochastic approximation with averaging innovation applied to
  finance.
\newblock {\em Monte Carlo Methods and Applications}, 18(1):1--51, 2012.

\bibitem{ljung1977}
L. Ljung.
\newblock Analysis of recursive stochastic algorithms.
\newblock {\em IEEE Transactions on Automatic Control}, 22(4):551--575, 1977.

\bibitem{leung}
T. Leung and X. Li.
\newblock \emph{Optimal Mean Reversion Trading.}
\newblock World Scientific, 2015.

\bibitem{robbins1951stochastic}
H. Robbins and S. Monro.
\newblock A stochastic approximation method.
\newblock {\em Annals of Mathematical Statistics}, 22(3):400--407, 1951.

\bibitem{sacks}
J. Sacks.
\newblock Asymptotic distribution of
stochastic approximation procedures.
\newblock {\em Annals of Mathematical Statistics},
29:373--405, 1958.









\bibitem{www1} https://tradergav.com/what-is-mean-reversion-trading-strategy/

\bibitem{www2} https://www.warriortrading.com/mean-reversion/

\bibitem{www3} https://tradingstrategyguides.com/mean-reversion-trading-strategy/


\end{thebibliography}
\end{document}